[1994/12/01]
\documentclass{ijmart-mod}
\chardef\bslash=`\\ 

\usepackage{bm}
\usepackage{graphicx}
\usepackage[breaklinks=true]{hyperref}
\usepackage{mathtools}
\usepackage{caption}
\usepackage{amsmath}
\usepackage{amssymb}
\usepackage{array}
\usepackage{multirow}
\usepackage{soul}
\usepackage{xcolor}

\DeclareSymbolFont{symbolsC}{U}{txsyc}{m}{n}
\DeclareMathSymbol{\varparallelinv}{\mathrel}{symbolsC}{10}

\newtheorem{thm}{Theorem}[section]

\newtheorem{lem}[thm]{Lemma}
\newtheorem{prop}[thm]{Proposition}

\theoremstyle{definition}
\newtheorem{defn}[thm]{Definition}
\newtheorem{rem}[thm]{Remark}
\newtheorem{qtn}[thm]{Question}

\theoremstyle{remark}

\newcommand{\quotient}[2]{{\raisebox{.2em}{$#1$}\big/\raisebox{-.2em}{$#2$}}}

\newcommand{\contract}{\mathord{\varparallelinv}} 

\newcommand{\eval}[2][\right]{\relax
  \ifx#1\right\relax \left.\fi#2#1\rvert}

\begin{document}
\title{Flip-graphs of non-orientable filling surfaces}

\author[P. Panda]{Pallavi Panda}
\address{Universit{\'e} Paris 13, Villetaneuse, France}
\email{pallavi.panda@univ-paris13.fr} 

\author[H. Parlier]{Hugo Parlier}
\address{University of Fribourg, Fribourg, Switzerland}
\email{hugo.parlier@unifr.ch} 

\author[L. Pournin]{Lionel Pournin}
\address{Universit{\'e} Paris 13, Villetaneuse, France}
\email{lionel.pournin@univ-paris13.fr}

\begin{abstract}
Consider a surface $\Sigma$ with punctures that serve as marked points and at least one marked point on each boundary component. We build a filling surface $\Sigma_n$ by singling out one of the boundary components and denoting by $n$ the number of marked points it contains.  We consider the triangulations of $\Sigma_n$ whose vertices are the marked points and the associated flip-graph $\mathcal{F}(\Sigma_n)$. Quotienting $\mathcal{F}(\Sigma_n)$ by the homeomorphisms of $\Sigma$ that fix the privileged boundary component results in a finite graph $\mathcal{MF}(\Sigma_n)$. Bounds on the diameter of $\mathcal{MF}(\Sigma_n)$ are available when $\Sigma$ is orientable and we provide corresponding bounds when $\Sigma$ is non-orientable. We show that the diameter of this graph grows at least like $5n/2$ and at most like $4n$ as $n$ goes to infinity. If $\Sigma$ is an unpunctured M{\"o}bius strip, $\mathcal{MF}(\Sigma_n)$ coincides with $\mathcal{F}(\Sigma_n)$ and we prove that the diameter of this graph grows exactly like $5n/2$ as $n$ goes to infinity.
\end{abstract}


\maketitle

\section{Introduction}\label{NOFG.sec.1}

Consider a finite-type topological surface $\Sigma$. We assume that $\Sigma$ has at least one boundary component, and among these, we choose one of them to be the \emph{privileged} boundary component of $\Sigma$. Let us now select a finite subset $\mathcal{P}$ of points from $\Sigma$ in such a way that each boundary component of $\Sigma$ contains at least one point from $\mathcal{P}$. We refer to the points in $\mathcal{P}$ that belong to the interior of $\Sigma$ as \emph{punctures} and to all other points in $\mathcal{P}$ as \emph{boundary points}. We denote by $n$ the number of points from $\mathcal{P}$ on the privileged boundary component and by $\Sigma_n$ the resulting surface equipped with these points. Here, we think of $\Sigma_n$ as a surface where $n$ can vary but whose topology (genus, number of boundary components, orientability) and the points in $\mathcal{P}$ that do not belong to the privileged boundary are otherwise fixed. This is called a \emph{filling surface} in~\cite{ParlierPournin2017} because such a surface fills the privileged boundary component.

A triangulation $T$ of $\Sigma_n$ is an inclusion-wise maximal set of pairwise non-crossing and non-homotopic simple arcs in $\Sigma_n$ between points in $\mathcal{P}$, considered up to isotopy. Even though $T$ is not a necessarily triangulation in the simplicial sense (see Figure \ref{NOFG.sec.1.fig.1} for example), the assumption that it is inclusion-wise maximal implies that the arcs in $T$ decompose $\Sigma_n$ into triangles: removing these arcs from $\Sigma$ results in a collection of open disks bounded by exactly three arcs in $T$, which we think of as triangles. 
\begin{figure}
\begin{centering}
\includegraphics[scale=1]{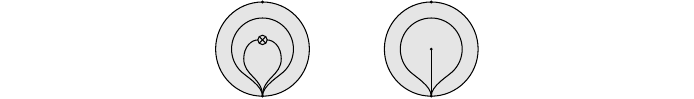}
\caption{A triangulation of the M{\"o}bius strip with two marked points in the boundary (shown in the cross-cap model of the M{\"o}bius strip) and a triangulation of the once-punctured disk with two marked points in the boundary.}\label{NOFG.sec.1.fig.1}
\end{centering}
\end{figure}
The set of the triangulations of $\Sigma_n$ can be given a structure as follows. Consider a triangulation $T$ of $\Sigma_n$ and an arc $\alpha$ in $T$ that is incident to two distinct triangles of $T$. Replacing $\alpha$ in $T$ with the other diagonal $\alpha'$ of the quadrilateral obtained by gluing these two triangles results in a different triangulation $T'$ of $\Sigma_n$. The move from $T$ to $T'$ is called a \emph{flip}. Equivalently, we say that $T'$ is obtained from $T$ by flipping $\alpha$. When an arc does not bound two distinct triangles, it cannot be flipped; this happens only when a loop arc bounds a punctured disk or a M{\"o}bius strip without punctures as shown in Figure \ref{NOFG.sec.1.fig.1}: the unique arc surrounded by the loop is contained in the boundary of a single triangle. Note that flips are reversible operations as one recovers $T$ by flipping $\alpha'$ in $T'$. This operation allows to consider the \emph{flip-graph} $\mathcal{F}(\Sigma_n)$ of $\Sigma_n$: the vertices of this graph are the triangulations of $\Sigma_n$ and there is an edge of $\mathcal{F}(\Sigma_n)$ between two triangulations when these triangulations can be changed into one another by a single flip.

The geometry of $\mathcal{F}(\Sigma_n)$ is particularly interesting because, thanks to the Schwarz--Milnor lemma, it is a quasi-isometric model for the mapping class group of $\Sigma_n$ \cite{DisarloParlier2019}. In the case when $\Sigma$ is a disk without punctures, the mapping class group of $\Sigma_n$ is trivial but the geometry of $\mathcal{F}(\Sigma_n)$ has been intensively studied \cite{AddarioBerryReedScottWood2018,Pournin2014,Pournin2019,PourninWang2021,SleatorTarjanThurston1988} because in this special case, $\mathcal{F}(\Sigma_n)$ is the $1$\nobreakdash-skeleton of the associahedron \cite{Lee1989,Stasheff1963a,Stasheff1963b,Tamari1951}, a polytope that appears in many areas of science. When $\Sigma$ is an arbitrary surface, it is known that $\mathcal{F}(\Sigma_n)$ is always connected~\cite{Bell2021,DisarloParlier2019,Hatcher1991,Mosher1995} and when $\Sigma$ is orientable, asymptotic estimates are known for the number of geodesic paths in $\mathcal{F}(\Sigma_n)$ between two triangulations \cite{ParlierPournin2025} but this graph is infinite except for a few surfaces $\Sigma$. However, considering the group $\mathrm{Mod}(\Sigma_n)$ of the homeomorphisms up to isotopy of $\Sigma_n$ that preserve the privileged boundary component pointwise
, the quotient
$$
\mathcal{MF}(\Sigma_n)=\quotient{\mathcal{F}(\Sigma_n)}{\mathrm{Mod}(\Sigma_n)}
$$
is a finite, connected flip-graph, the \emph{modular flip-graph} of $\Sigma_n$ whose geometry is also interesting \cite{DisarloParlier2019,ParlierPournin2017}. 
For instance, the diameter of $\mathcal{MF}(\Sigma_n)$ helps quantify the quasi-isometry that we mention above \cite{DisarloParlier2019}. 

It is shown in \cite{ParlierPournin2017} that, for any orientable filling surface $\Sigma$ such that the topology of $\Sigma$ and the points of $\mathcal{P}$ that do not belong to the privileged boundary component are fixed, there exists a constant $c_\Sigma$ satisfying
\begin{equation}\label{NOFG.sec.1.eq.1}
\lim_{n\rightarrow\infty}\frac{\mathrm{diam}(\mathcal{MF}(\Sigma_n))}{n}=c_\Sigma\mbox{.}
\end{equation}

Bounds on $c_\Sigma$ are given in~\cite{ParlierPournin2017,ParlierPournin2018a}. In particular,
$$
2\leq{c_\Sigma}\leq4
$$
for every orientable filling surface $\Sigma$. The lower bound is sharp as $c_\Sigma$ is equal to $2$ when $\Sigma$ is a disk or a once-punctured disk~\cite{ParlierPournin2018b,Pournin2014}. Moreover, it is shown in~\cite{ParlierPournin2017} that $c_\Sigma$ is equal to $5/2$ when $\Sigma$ is a cylinder without punctures, and to $3$ when $\Sigma$ is a three-holed sphere without punctures.

An interesting subclass of the filling surfaces is formed by the \emph{one-holed surfaces}, filling surfaces without punctures and a single boundary component (that necessarily serves as the privileged boundary component). If $\Sigma$ is a genus $g$, orientable one-holed surface, then it is known that 
$$
c_\Sigma\leq4-\frac{1}{4g}
$$
when $g$ is at least $2$ and that $c_\Sigma$ is at most $23/8$ when $g$ is equal to $1$ \cite{ParlierPournin2018a}.

The first purpose of this article is to extend a number of the results found in \cite{ParlierPournin2017,ParlierPournin2018a} to the non-orientable case. We will see in particular that if $\Sigma$ is a non-orientable filling surface, then there still exists a constant $c_\Sigma$ satisfying (\ref{NOFG.sec.1.eq.1}) and that this constant can be bounded as follows. 

\begin{thm}\label{NOFG.sec.1.thm.1}
If $\Sigma$ is a non-orientable filling surface, then
$$
\frac{5}{2}\leq{c_\Sigma}\leq4\mbox{.}
$$

Moreover, if $\Sigma$ is a demigenus $g$, non-orientable one-holed surface, then 
$$
c_\Sigma\leq4-\frac{1}{2g}
$$
when $g$ is at least $3$ and $c_\Sigma$ is at most $23/8$ when $g$ is equal to $2$.
\end{thm}

In the statement of Theorem \ref{NOFG.sec.1.thm.1}, the \emph{demigenus} of a non-orientable one-holed surface $\Sigma$ refers to the number of cross-caps that need to be inserted within a topological disk in order to recover $\Sigma$ or, equivalently, to the least number of arcs that need to be removed from $\Sigma$ in order to transform it into a disk.

The M{\"o}bius strip $\mathrm{M}$ without punctures has a single boundary and it is the simplest example of a non-orientable one-holed (and therefore filling) surface. Recall that all the homeomorphisms of $\mathrm{M}$ that preserve the boundary pointwise coincide up to isotopy. In other words, the (pure) mapping class group of $\mathrm{M}$ is trivial and as a consequence, $\mathcal{MF}(\mathrm{M}_n)$ coincides with $\mathcal{F}(\mathrm{M}_n)$. The second purpose of this article is to estimate the diameter of $\mathcal{F}(\mathrm{M}_n)$ and to show that the lower bound on $c_\Sigma$ stated by Theorem \ref{NOFG.sec.1.thm.1} is achieved by $\mathrm{M}$.

\begin{thm}\label{NOFG.sec.1.thm.2}
For every positive $n$,
$$
\biggl\lfloor\frac{5}{2}n\biggr\rfloor-2\leq\mathrm{diam}(\mathcal{F}(\mathrm{M}_n))\leq\biggl\lfloor\frac{5}{2}n\biggr\rfloor\mbox{.}
$$
\end{thm}


As mentioned above, the triangulations of $\mathrm{M}_n$ do not necessarily form simplicial complexes. In particular, the two endpoints of an arc can coincide, two distinct arcs can have the same pair of extremities, and two edges of a triangle can be formed by a single arc (see Figure~\ref{NOFG.sec.1.fig.1}). Paul Edelman and Victor Reiner consider the subgraph $\mathcal{F}_\star(\mathrm{M}_n)$ induced in $\mathcal{F}(\mathrm{M}_n)$ by the triangulations that form simplicial complexes and ask for its diameter in \cite{EdelmanReiner1997}. The third purpose of this article is to provide bounds on this diameter.
\begin{thm}\label{NOFG.sec.1.thm.3}
There exists a constant $K$ such that for all $n$ at least $5$, 
$$
\biggl\lfloor\frac{5}{2}n\biggr\rfloor-16\leq\mathrm{diam}(\mathcal{F}_\star(\mathrm{M}_n))\leq4n+K\mbox{.}
$$
\end{thm}

We ask the following question.

\begin{qtn}
Consider an orientable or non-orientable one-holed surface $\Sigma$. If $\Sigma$ is orientable, is $c_\Sigma$ a monotonic function of the genus of $\Sigma$ and if $\Sigma$ is non-orientable, is $c_\Sigma$ a monotonic function of the demigenus of $\Sigma$?
\end{qtn}

The article is organized as follows. In Section~\ref{NOFG.sec.2}, we explain how the existence of $c_\Sigma$ and the upper bounds on that number can be proven in the non-orientable case. The argument is similar to the one in \cite{ParlierPournin2017,ParlierPournin2018a} and only the differences are highlighted. In Section~\ref{NOFG.sec.3}, we prove a general lower bound on the diameter on $\mathcal{MF}(\Sigma_n)$ when $\Sigma$ is a filling surface that is either non-orientable or orientable but of positive genus (see Theorem~\ref{NOFG.sec.3.thm.1}). While we use the same tools as in~\cite{ParlierPournin2018a}, the proof is significantly more general. This bound implies both the lower bound on $c_\Sigma$ from Theorem \ref{NOFG.sec.1.thm.1} and the lower bound on $\mathrm{diam}(\mathcal{F}(\mathrm{M}_n))$ from Theorem~\ref{NOFG.sec.1.thm.2}. In Section~\ref{NOFG.sec.4}, we establish a structural property of the triangulations of $\mathrm{M}_n$ and use it to show the upper bound on $\mathrm{diam}(\mathcal{F}(\mathrm{M}_n))$ from Theorem~\ref{NOFG.sec.1.thm.2}. Finally, we prove Theorem~\ref{NOFG.sec.1.thm.3} in Section~\ref{NOFG.sec.5}.

\section{From orientable to non-orientable surfaces}\label{NOFG.sec.2}

In this section, we explain how the proofs from \cite{ParlierPournin2017,ParlierPournin2018a} can be adapted to the non-orientable case. We first recall some notation and terminology. A path in $\mathcal{F}(\Sigma_n)$ from a triangulation $T^-$ to a triangulation $T^+$ will be represented as the sequence $(T_i)_{0\leq{i}\leq{k}}$ of triangulations of $\Sigma_n$ along it. In particular, $T_0$ coincides with $T^-$ and $T_k$ with $T^+$. Moreover, when $1\leq{i}\leq{k}$, triangulations $T_{i-1}$ and $T_i$ are related by a flip. This path has length $k$ but there may be paths of different lengths between $T^-$ to $T^+$. The \emph{geodesic paths} (or for short the \emph{geodesics}) between $T^-$ and $T^+$ in $\mathcal{F}(\Sigma_n)$ will be the ones whose length is the distance of these triangulations in $\mathcal{F}(\Sigma_n)$. We will use the same notation and terminology for the paths in $\mathcal{MF}(\Sigma_n)$, the only difference being that the triangulations are considered up to the homeomorphisms in $\mathrm{Mod}(\Sigma_n)$. We denote the distance of $T^-$ and $T^+$ in both flip-graphs by $d(T^-,T^+)$ as it will always be clear from the context which flip-graph this distance is measured in.

The following strong convexity property of $\mathcal{F}(\Sigma_n)$ is proven in \cite{DisarloParlier2019} (see Theorem 1.1 therein) in the case when $\Sigma$ is orientable.

\begin{thm}\label{NOFG.sec.2.thm.1}
Consider an orientable filling surface $\Sigma$ and two triangulations $T^-$ and $T^+$ of $\Sigma_n$. The arcs common to $T^-$ and $T^+$ are contained in every triangulation along every geodesic between $T^-$ and $T^+$ in $\mathcal{F}(\Sigma_n)$.  
\end{thm}

While this strong convexity property fails in general in the case of the modular flip-graph as pointed out in \cite{ParlierPournin2017}, it remains true for \emph{certain} of the arcs common to two triangulations of $\Sigma_n$. We call an arc $\alpha$ in $\Sigma_n$ \emph{parallel to the privileged boundary} when $\Sigma_n \mathord{\setminus}\alpha$ has two connected components and one of them is an unpunctured disk. Since the homemorphisms in $\mathrm{Mod}(\Sigma_n)$ fix the privileged boundary pointwise, the arcs parallel to the privileged boundary are invariant under all of these homeomorphisms and that strong convexity property stated by Theorem~\ref{NOFG.sec.2.thm.1} carries over to $\mathcal{MF}(\Sigma_n)$ for these arcs. This can be alternatively recovered using the proof of \cite[Lemma~3]{SleatorTarjanThurston1988} almost word for word, precisely because the homeomorphisms in $\mathrm{Mod}(\Sigma_n)$ fix the arcs parallel to the privileged boundary. In fact, the proof of \cite[Lemma~3]{SleatorTarjanThurston1988} also works as is for these arcs when $\Sigma$ is non-orientable and we get the following.
\begin{thm}\label{NOFG.sec.2.cor.1}
Consider a possibly non-orientable filling surface $\Sigma$ and two triangulations $T^-$ and $T^+$ of $\Sigma_n$. The arcs parallel to the privileged boundary of $\Sigma$ and common to $T^-$ and $T^+$ are contained in every triangulation along every geodesic between $T^-$ and $T^+$ in $\mathcal{MF}(\Sigma_n)$.  
\end{thm}

An arc contained in a boundary component of a (possibly non-orientable) filling surface $\Sigma$ whose extremities are consecutive points of $\mathcal{P}$ in that boundary component will be called a \emph{boundary arc} of $\Sigma_n$. All the boundary arcs of $\Sigma_n$ are contained in every triangulation of $\Sigma_n$ and for these reason they will also be referred to as the boundary arcs of these triangulations. If $T$ is a triangulation of $\Sigma_n$, we will call every arc of $T$ that is not a boundary arc an \emph{interior arc} of $T$. An \emph{ear} of $T$ will be a vertex $v$ of $T$ that is not incident to any interior arc of $T$. If $v$ belongs to the privileged boundary, then the edge $\alpha$ of $t$ that is not incident to $v$ is parallel to the privileged boundary. In particular, according to Theorem~\ref{NOFG.sec.2.cor.1}, the triangulations of $\Sigma_n$ that admit $v$ as an ear induce a strongly convex subgraph of $\mathcal{MF}(\Sigma_n)$. By construction, this subgraph is a copy of $\mathcal{MF}(\Sigma_{n-1})$ and we immediately obtain the following statement.

\begin{prop}\label{NOFG.sec.2.prop.1}
Consider a possibly non-orientable filling surface $\Sigma$. If $n$ is at least $2$, then $\mathrm{diam}(\mathcal{MF}(\Sigma_n))\geq\mathrm{diam}(\mathcal{MF}(\Sigma_{n-1}))$
\end{prop}

According to \cite[Theorem 3.1]{ParlierPournin2017}, for any orientable filling surface $\Sigma$,
$$
\mathrm{diam}(\mathcal{MF}(\Sigma_n))\leq4n+K_\Sigma
$$
where $K_\Sigma$ is a constant that depends on $\Sigma$ but not on $n$. Together with the orientable case of Proposition \ref{NOFG.sec.2.prop.1} this yields the existence of the constant $c_\Sigma$ that satisfies (\ref{NOFG.sec.1.eq.1}). The proof of \cite[Theorem 3.1]{ParlierPournin2017} relies on the following lemma that is established in \cite{ParlierPournin2017} in the orientable case, but whose proof can be immediately transposed to the non-orientable case. For this reason, it will be omitted but note that a refined statement will be established in Section \ref{NOFG.sec.5} using the same strategy in order to upper bound the diameter of $\mathcal{F}_\star(\mathrm{M}_n)$.

\begin{thm}[{\cite[Lemma 3.2]{ParlierPournin2017}}]\label{NOFG.sec.2.thm.2}
Consider a possibly non-orientable filling surface $\Sigma$, two triangulations $T^-$ and $T^+$ of $\Sigma_n$, and a vertex $v$ of these triangulations in the privileged boundary of $\Sigma$. If $n$ is at least $2$ and the number of interior arcs incident to $v$ in $T^-$ and in $T^+$ sums to at most $4$, then there exist two triangulations $\widetilde{T}^-$ and $\widetilde{T}^+$ of $\Sigma_n$ admitting $v$ as an ear such that
$$
d(T^-,\widetilde{T}^-)+d(T^+,\widetilde{T}^+)\leq4
$$
where both distances are measured either in $\mathcal{F}(\Sigma_n)$ or in $\mathcal{MF}(\Sigma_n)$.
\end{thm}

Now observe that, by a simple Euler characteristic argument, if $\kappa(\Sigma_n)$ denotes the number of interior arcs in a triangulation of $\Sigma_n$, then
$$
\lim_{n\rightarrow\infty}\frac{\kappa(\Sigma_n)}{n}=1\mbox{.}
$$

As a consequence, for any (possibly non-orientable) filling surface $\Sigma$, for any sufficiently large $n$ and any two triangulations $T^-$ and $T^+$ of $\Sigma_n$, there must exist a vertex $v$ of these triangulations in the privileged boundary of $\Sigma$ such that the number of interior arcs incident to $v$ in $T^-$ and in $T^+$ sums at most $4$. It therefore follows from Theorem \ref{NOFG.sec.2.thm.2} that for any sufficiently large $n$,
$$
\mathrm{diam}(\mathcal{MF}(\Sigma_n))\leq4+\mathrm{diam}(\mathcal{MF}(\Sigma_{n-1}))
$$
which immediately proves the following.
\begin{thm}\label{NOFG.sec.2.thm.3}
For any, possibly non-orientable, filling surface $\Sigma$ there exists a constant $K_\Sigma$ that depends on $\Sigma$ but not on $n$ such that for all positive $n$,
$$
\mathrm{diam}(\mathcal{MF}(\Sigma_n))\leq4n+K_\Sigma\mbox{.}
$$
\end{thm}

Combining Proposition \ref{NOFG.sec.2.prop.1} with Theorem \ref{NOFG.sec.2.thm.3} shows that for every non-orientable filling surface $\Sigma$, there exists a constant $c_\Sigma$ satisfying (\ref{NOFG.sec.1.eq.1}). In turn, the first upper bound on $c_\Sigma$ stated by Theorem \ref{NOFG.sec.1.thm.1} follows from Theorem \ref{NOFG.sec.2.thm.3}. Let us now explain how the upper bound on $c_\Sigma$ provided by the same theorem in the special case when $\Sigma$ is one-holed can be recovered from the argument in~\cite{ParlierPournin2018a}. The demigenus of a non-orientable one-holed surface $\Sigma$ is the number of cross-caps that need to be inserted within a disk in order to build $\Sigma$. This number is also the least number of arcs that need to be removed from $\Sigma$ in order to transform it into a disk. The latter definition also makes sense in the orientable case and we will then call this number (which is in fact twice the surface's genus) the demigenus of an orientable one-holed surface. Note that a one-holed surface is non-orientable if and only if its demigenus is odd.

Consider a demigenus $g$ possibly non-orientable one-holed surface $\Sigma$. In order to upper bound the diameter of $\mathcal{MF}(\Sigma_n)$ we shall build an explicit path in $\mathcal{MF}(\Sigma_n)$ between any two triangulations $T^-$ and $T^+$ of $\Sigma_n$ using the same strategy as for Theorems 4 and 5 from \cite{ParlierPournin2018a}. First consider $T^-$. One can find a set of arcs $\gamma^-_1$ to $\gamma^-_g$ in $T^-$ such that the surface
$$
\Sigma_n\mathord{\setminus}\bigcup_{i=1}^g\gamma_i^-
$$
is a disk with $n+2g$ marked points on the boundary that we will denote by $\Delta^-_{n+2g}$. Each of the arcs $\gamma_i$ has two copies that serve as boundary arcs of $\Delta ^-_{n+2g}$. Two such consecutive arcs along the boundary of $\Delta^-_{n+2g}$ may be incident but if they are not, a sequence of boundary arcs of $\Sigma_n$ lies between them. Moreover, all the boundary arcs of $\Sigma_n$ belong to exactly one such sequence.

One can do the same using an appropriate set of arcs $\gamma^+_1$ to $\gamma^+_g$ in $T^+$ to build a disk $\Delta^+_{n+2g}$. By construction, there exists a sequence of $\lceil{n/(4g)}\rceil$ boundary arcs of $\Sigma_n$ that are consecutive along the boundaries of both $\Delta^-_{n+2g}$ and $\Delta^+_{n+2g}$. We denote by $v$ one of the extremities of the arcs in this sequence. Now observe that if the $n+2g-3$ interior arcs of $T^-$ other than $\gamma^-_1$ to $\gamma^-_n$ are not all incident to $v$, then one can always perform a flip in $T^-$ that introduces a new arc incident to $v$. This comes from the fact that these arcs are precisely the interior arcs in the triangulation of $\Delta^-_{n+2g}$ induced by $T^-$. Starting from $T^-$, one can therefore reach a triangulation of $\Sigma_n$ that contains $\gamma^-_1$ to $\gamma^-_g$ and whose other interior arcs are all incident to $v$ in at most $n+2g-3$ flips. Further flipping $\gamma^-_1$ to $\gamma^-_g$ in this triangulation introduces $g$ loops $\delta^-_1$ to $\delta^-_g$ twice incident to $v$ and we denote by $U^-$ the resulting triangulation. Let $U^+$ be the triangulation obtained using the similar sequence of flips starting from $T^+$. All the interior arcs of $U^+$ are incident to $v$ and among these arcs, we denote by $\delta^+_1$ to $\delta^+_g$ the loops that have been introduced when flipping $\gamma^+_1$ to $\gamma^+_g$.

The following is shown in \cite{ParlierPournin2018a} (see Theorems 4 and 5 therein) in the case when $g$ is even. However, the argument does not depend on the parity of $g$ and immediately carries over to the non-orientable case.

\begin{lem}\label{NOFG.sec.2.lem.1}
The distance of $U^-$ and $U^+$ in $\mathcal{MF}(\Sigma_n)$ is at most
$$
\biggl(2-\frac{1}{2g}\biggr)n+2g-1+\mathrm{diam}(\mathcal{MF}(\Sigma_1))
$$
when $g$ is at least $3$ and at most $7n/8+6$ when $g$ is equal to $2$.
\end{lem}

Note that, when $g$ is equal to $2$, there is an homeomorphism in $\mathrm{Mod}(\Sigma_n)$ that sends $\gamma^-_1$ to $\gamma^+_1$ and $\gamma^-_2$ to $\gamma^+_2$. This explains why Lemma \ref{NOFG.sec.2.lem.1} provides a slightly better bound in that case. Since by construction
\begin{equation}\label{NOFG.sec.2.eq.1}
d(T^-,U^-)+d(T^+,U^+)\leq 2n+6g-6\mbox{,}
\end{equation}
one obtains the following as a consequence of Lemma \ref{NOFG.sec.2.lem.1}.
\begin{thm}\label{NOFG.sec.2.thm.3}
For any demigenus $g$ possibly non-orientable one-holed surface $\Sigma$, there exists a constant $K_g$ that depends on $g$ but not on $n$ such that
$$
\mathrm{diam}(\mathcal{MF}(\Sigma_n))\leq\biggl(4-\frac{1}{2g}\biggr)n+K_g\mbox{.}
$$
when $g$ is at least $3$ and
$$
\mathrm{diam}(\mathcal{MF}(\Sigma_n))\leq\frac{23}{8}n+K_2
$$
when $g$ is equal to $2$.
\end{thm}

By Lemma \ref{NOFG.sec.2.lem.1} and (\ref{NOFG.sec.2.eq.1}) we can take\footnote{In \cite{ParlierPournin2018a}, the $6g$ in the right-hand side of (\ref{NOFG.sec.2.eq.1}) is incorrectly replaced by $8g$ but this only affects the expression of $K_g$ and not the final result. Moreover, the value of $K_2$ is incorrectly set to $8$ instead of $12$ in \cite[Theorem 1]{ParlierPournin2018a} and \cite[Theorem 4]{ParlierPournin2018a} because the flips of arcs $\gamma^-_i$ and $\gamma^+_i$ have not been counted in the proof of the latter theorem.} $K_2$ equal to $12$ and
$$
K_g=8g-7+\mathrm{diam}(\mathcal{MF}(\Sigma_1))
$$
when $g$ is at least $3$ in the statement of Theorem \ref{NOFG.sec.2.thm.3}. The upper bounds on $c_\Sigma$ stated by Theorem \ref{NOFG.sec.1.thm.1} in the case when $\Sigma$ is a non-orientable one-holed surface is immediately obtained as a consequence of Theorem \ref{NOFG.sec.2.thm.3}.

\section{Far apart triangulations of non-orientable filling surfaces}\label{NOFG.sec.3}

It is shown in \cite{ParlierPournin2018a} that for any positive genus, oriented one-holed surface $\Sigma$,
\begin{equation}\label{NOFG.sec.3.eq.0}
\mathrm{diam}(\mathcal{MF}(\Sigma_n))\geq\biggl\lfloor\frac{5}{2}n\biggr\rfloor-2\mbox{.}
\end{equation}

The aim of this section is to extend this result to non-orientable filling surfaces and to strengthen it as follows in the case of one-holed surfaces.

\begin{thm}\label{NOFG.sec.3.thm.1}
If $\Sigma$ is a non-orientable filling surface or a positive genus orientable filling surface, then it satisfies (\ref{NOFG.sec.3.eq.0}) and if $\Sigma$ is a non-orientable or orientable one-holed surface other than a disk, then it satisfies
\begin{equation}\label{NOFG.sec.3.thm.1.eq.1}
\mathrm{diam}(\mathcal{MF}(\Sigma_n))\geq\biggl\lfloor\frac{5}{2}n\biggr\rfloor-2+\mathrm{diam}(\mathcal{MF}(\Sigma_1))\mbox{.}
\end{equation}
\end{thm}

The proof of Theorem \ref{NOFG.sec.3.thm.1} will use a similar strategy as that of Theorem~7 from~\cite{ParlierPournin2018a} but in a more general setting. This level of generality requires considering more complicated families of triangulation pairs.

As in \cite{ParlierPournin2017,ParlierPournin2018a,ParlierPournin2018b,ParlierPournin2025,Pournin2014}, we will use \emph{boundary arc contractions}. Consider a filling surface $\Sigma$ and a triangulation $T$ of $\Sigma_n$. Provided that $n$ is at least $2$, the boundary arcs of $\Sigma_n$ contained (up to isotopy) in the privileged boundary have two distinct vertices. Consider such an arc $\alpha$. This arc naturally belongs to $T$ and there is a triangle $t$ of $T$ incident to it. Now remove the interiors of $\alpha$ and of $t$ from $\Sigma_n$ and glue the two remaining edges of $t$ to one another in such a way that the two extremities of $\alpha$ are identified as shown in Figure~\ref{NOFG.sec.3.fig.2.5}. This operation, which we refer to as the contraction of $\alpha$ in $T$ results in a triangulation of $\Sigma_{n-1}$ that we denote by $T\contract\alpha$. Now consider a triangulation $T'$ of $\Sigma_n$ that is obtained from $T$ by flipping some arc. We say that this flip is incident to $\alpha$ when it modifies the triangle incident to $\alpha$ within $T$. In other words, such a flip exchanges the diagonals of a quadrilateral incident to $\alpha$. The following is established in \cite{ParlierPournin2017} (see Theorem 2.4 therein) in the case when $\Sigma$ is orientable but its proof works as is in the non-orientable case.
\begin{figure}
\begin{centering}
\includegraphics[scale=1]{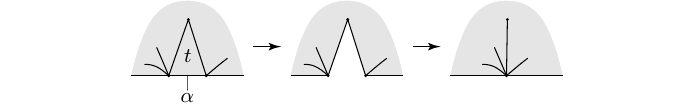}
\caption{The contraction of $\alpha$.}\label{NOFG.sec.3.fig.2.5}
\end{centering}
\end{figure}

\begin{lem}\label{NOFG.sec.3.lem.1}
Consider a possibly non-orientable filling surface $\Sigma$, a boundary arc $\alpha$ of $\Sigma_n$ contained up to isotopy in the privileged boundary of $\Sigma$, and two triangulations $T^-$ and $T^+$ of $\Sigma_n$. If $n$ is at least $2$ and $f$ flips are incident to $\alpha$ along some geodesic from $T^-$ to $T^+$ in $\mathcal{MF}(\Sigma_n)$, then 
$$
d(T^-,T^+)\geq{d(T^-\contract\alpha,T^+\contract\alpha)+f}\mbox{.}
$$
\end{lem}

Let us assume from now on that $\Sigma$ is either a non-orientable filling surface or a positive genus orientable filling surface. If $\Sigma$ is not one-holed, then we cut it along an arc whose extremities are different points in the privileged boundary into a one-holed filling surface $\Sigma^\otimes$ and a genus $0$ orientable filling surface $\Sigma^\odot$ (see Figure \ref{NOFG.sec.3.fig.1}). 
\begin{figure}[b]
\begin{centering}
\includegraphics[scale=1]{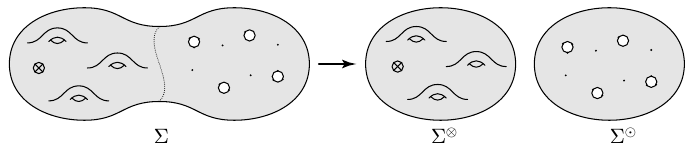}
\caption{If $\Sigma$ is not one holed we cut it into $\Sigma^\otimes$ and $\Sigma^\odot$.}\label{NOFG.sec.3.fig.1}
\end{centering}
\end{figure}
If $\Sigma$ is one-holed then we use $\Sigma^\otimes$ as an alternative notation for $\Sigma$. Since $\Sigma$ is either non-orientable or has positive genus, $\Sigma^\otimes$ cannot be a disk. We will build two triangulations families $A_n^-$ and $A_n^+$ of $\Sigma_n$ and show that their distance is at least the lower bound stated by Theorem~\ref{NOFG.sec.3.thm.1}.

Consider a triangulation $X_1^-$ of $\Sigma_1^\otimes$ and a triangulation $X_1^+$ of $\Sigma_1$. 
If $\Sigma$ is not one-holed, further consider a triangulation $O_1$ of $\Sigma_1^\odot$. A crucial point here is the existence of triangulations of $\Sigma_1^\otimes$ and, when $\Sigma$ is not one-holed, of $\Sigma_1^\odot$: such triangulations exist precisely because these filling surfaces are not disks. The triangulations  $A_n^-$ and $A_n^+$ are obtained from $X_1^-$, $X_1^+$, and $O_1$ as shown in Figure \ref{NOFG.sec.3.fig.2}. The triangulation $A_n ^+$ is shown in the third row in this figure depending on whether $n$ is even or odd. It is composed from a zigzag triangulation at one end of which there is a loop arc bounding $X_1^+$ and surrounded by two arcs with the same vertex pair. One of these arcs will be denoted by $\alpha'$ and the boundary arc at the opposite end of the zigzag is denoted by $\alpha$. When $\Sigma$ is one-holed, $A_n^-$ is represented in the top row of Figure \ref{NOFG.sec.3.fig.2}. The construction is similar except that the loop arc, which we denote by $\varepsilon$ is at the other end of the zigzag. When $\Sigma$ is not one-holed, $A_n^-$ is represented in the central row of Figure \ref{NOFG.sec.3.fig.2} and contains a loop arc at both ends of the zigzag. The figure really determines $A_n^-$ and $A_n^+$ when $n$ is at least $2$. We set $A_1^-$ and $A_1^+$ to be the triangulations obtained from $A_2^-$ and $A_2^+$ by contracting $\alpha$.

It will be important to keep in mind that $A_n^-$ and $A_n^+$ depend on $X_1^-$, $X_1^+$, and $O_1$. In this perspective, $A_n^-$ and $A_n^+$ can be thought of as two families of triangulations of $\Sigma_n$. In order not to overburden notations, we will not mark this dependence and will only mention it when needed.

\begin{figure}[b]
\begin{centering}
\includegraphics[scale=1]{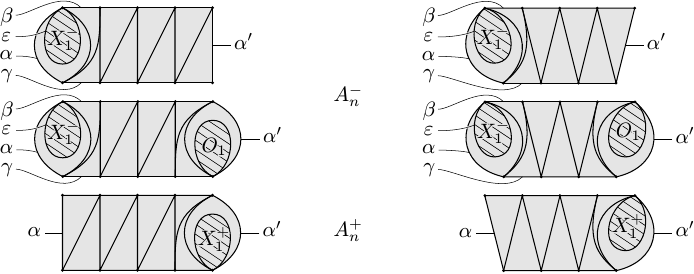}
\caption{The triangulations $A_n^-$ (first two rows) and $A_n^+$ (third row) when $n$ is even (left) and odd (right). The first row shows $A_n^-$ when $\Sigma$ is one-holed, and the second when it is not.}\label{NOFG.sec.3.fig.2}
\end{centering}
\end{figure}

These triangulations behave well with respect to certain boundary arc contractions. In particular, if $n$ is at least $2$, then
\begin{equation}\label{NOFG.sec.3.eq.0.5}
A_n^\pm\contract\alpha=A_{n-1}^\pm
\end{equation}
where the triangulations $A_n^\pm$ and $A_{n-1}^\pm$ that appear on both sides of this equality are built using the same initial triangulations $X_1^-$, $X_1^+$, and $O_1$. Moreover, if $n$ is at least $3$, then denote by $\beta$ and $\gamma$ the boundary arc of $\Sigma_n$ adjacent to $\alpha$ with the convention that $\beta$ is incident to $u$. Note that when $n$ is equal to $3$, $\gamma$ coincides with $\alpha'$. Contracting $\beta$ and $\gamma$ in $A_n^-$ or $A_n^+$ just removes two opposite boundary arcs from the zigzag. As a consequence,
\begin{equation}\label{NOFG.sec.3.eq.1}
A_n^\pm\contract\beta\contract\gamma=A_{n-2}^\pm\mbox{.}
\end{equation}
where again, the triangulations $A_n^\pm$ and $A_{n-2}^\pm$ on both sides of the equality are built from the same $X_1^-$, $X_1^+$, and $O_1$. The proof then consists in showing that when $n$ is at least $2$, there are sufficiently many flips along any geodesic path between $A_n^-$ and $A_n^+$ in $\mathcal{MF}(\Sigma_n)$ that are incident to $\alpha$, $\beta$, or $\gamma$. This will allow to obtain a recursive lower bound on the distance of $A_n^-$ and $A_n^+$ in $\mathcal{MF}(\Sigma_n)$ via Lemma \ref{NOFG.sec.3.lem.1}. 
By that lemma, the following is immediate.

\begin{lem}\label{NOFG.sec.3.lem.0}
If $n$ is at least $2$ and at least three flips are incident to $\alpha^-$ along some geodesic path between $A_n^-$ and $A_n^+$ in $\mathcal{MF}(\Sigma_n)$, then
$$
d(A_n^-,A_n^+)\geq{d(A_{n-1}^-,A_{n-1}^+)+3}
$$
where $A_n^\pm$ and $A_{n-1}^\pm$ are built from the same triangulations $X_1^\pm$ and $O_1$.
\end{lem}

Let us now study the number of flips incident to $\alpha$ in the paths between $A_n^-$ and $A_n^+$ in $\mathcal{MF}(\Sigma_n)$. We have the following properties.
\begin{prop}\label{NOFG.sec.3.prop.2}
If $n$ is at least $2$, then at least two flips are incident to $\alpha$ along a path from $A_n^-$ to $A_n^+$ in $\mathcal{MF}(\Sigma_n)$ and if the first such flip removes $\varepsilon$, then at least three flips are incident to $\alpha$ along this path.
\end{prop}
\begin{proof}
Assume that $n$ is at least $2$ and consider a path from $A_n^-$ to $A_n^+$ in $\mathcal{MF}(\Sigma_n)$. First observe that if there is a unique flip incident to $\alpha$ along that path, then it must transform the triangle $t^-$ of $A_n^-$ incident to $\alpha$ into the triangle $t^+$ of $A_n^+$ incident to $\alpha$. These two triangles are shown on the left of Figure~\ref{NOFG.sec.3.fig.4} where the interior arc of $A_n^+$ that bounds $t^+$ is drawn as a dotted line. One can see in particular that this dotted arc crosses $\varepsilon$ twice. As $\varepsilon$ bounds $t^-$, it is not possible to transform $t^-$ into $t^+$ by a single flip. As a consequence, at least two flips are incident to $\alpha$ along the considered path.
\begin{figure}
\begin{centering}
\includegraphics[scale=1]{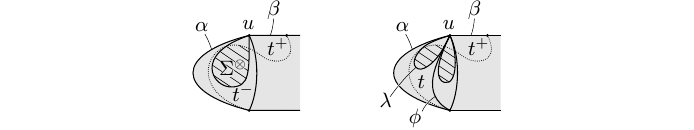}
\caption{The triangles $t^-$, $t^+$, and $t$.}\label{NOFG.sec.3.fig.4}
\end{centering}
\end{figure}

Now assume that the first flip incident to $\alpha$ along that path removes $\varepsilon$. Since $\Sigma^\otimes$ is one-holed, this flip necessarily introduces an arc $\phi$ incident to $u$ as shown on the right of Figure \ref{NOFG.sec.3.fig.4}. The triangle $t$ incident to $\alpha$ after that flip still admits a loop $\lambda$ twice incident to $u$ as one of its edges. One of the interior arcs bounding $t^+$ shown as a dotted line crosses $\lambda$ twice and therefore, at least two more flips must be incident to $\alpha$ along the considered path. As a consequence, that path cannot have just two flips incident to the arc $\alpha$.
\end{proof}

Observe that the triangle incident to $\alpha$ in $A_2^-$ is bounded by $\varepsilon$ and two boundary arcs. As a consequence, the first flip incident to $\alpha$ along a geodesic path from $A_2^-$ to $A_2^+$ in $\mathcal{MF}(\Sigma_2)$ necessarily removes $\varepsilon$. Moreover, when $\Sigma$ is one-holed, $A_1^-$ and $A_1^+$ coincide with $X_1^-$ and $X_1^+$. The following statement is therefore a consequence of Proposition \ref{NOFG.sec.3.prop.2}, Lemma \ref{NOFG.sec.3.lem.0}, and Equation (\ref{NOFG.sec.3.eq.0.5}).
\begin{prop}\label{NOFG.sec.3.prop.2.5}
If $\Sigma$ is not a one-holed surface, then
$$
d(A_2^-,A_2^+)\geq3
$$
and if $\Sigma$ is a one-holed surface, then
$$
d(A_2^-,A_2^+)\geq{d(X_1^-,X_1^+)+3}\mbox{.}
$$
where the triangulations $A_2^\pm$ are built from the triangulations $X_1^\pm$.
\end{prop}

We now focus on the geodesics from $A_n^-$ to $A_n^+$ in $\mathcal{MF}(\Sigma_n)$ along which exactly two flips are incident to $\alpha$. By Proposition \ref{NOFG.sec.3.prop.2}, the first such flip removes the the interior arc $\delta$ in $A_n^-$ whose pair of vertices is the same than $\alpha$. Denote by $\phi$ the arc introduced by this flip. We will examine three cases depending on which arcs of $A_n ^-$ are crossed by $\phi$. The first case we consider is when $\phi$ does not cross any other arc of $A_n^-$ than $\delta$. In that case, the flip that replaces $\delta$ by $\phi$ is necessarily the one shown on the left of Figure \ref{NOFG.sec.3.fig.5}.

\begin{lem}\label{NOFG.sec.3.lem.2}
Assume that $n$ is at least $2$ and consider a geodesic path from $A_n^-$ to $A_n^+$ in $\mathcal{MF}(\Sigma_n)$ along which exactly two flips are incident to $\alpha$. Further consider the arc $\phi$ introduced by the first of these two flips. If $\delta$ is the only interior arc of $A_n^-$ crossed by $\phi$ then
$$
d(A_n^-,A_n^+)\geq{d(A_n^-\contract\beta,A_n^+\contract\beta)+4}\mbox{.}
$$
\end{lem}
\begin{proof}
According to Proposition \ref{NOFG.sec.3.prop.2}, the flip that introduces $\phi$ must remove the arc $\delta$. Assume that $\phi$ does not cross any arc of $A_n^-$ other than $\delta$. In that case, the first and the second flips incident to $\alpha$ along the considered path must be as shown on the left and on the right of Figure \ref{NOFG.sec.3.fig.5}. In particular, $\phi$ (shown as a dotted line on the left of the figure) has the same pair of vertices as $\beta$. The second flip replaces $\phi$ with the arc $\eta$ shown as a dotted line on the right of the figure. Observe that both of these flips are incident to $\beta$.
\begin{figure}
\begin{centering}
\includegraphics[scale=1]{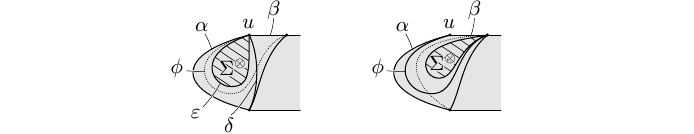}
\caption{Two flips incident to $\alpha$.}\label{NOFG.sec.3.fig.5}
\end{centering}
\end{figure}

One can see in Figure \ref{NOFG.sec.3.fig.5} that $\phi$ and $\eta$ are introduced in triangulations each containing a loop arc around $\Sigma^\otimes$. These loop arcs are twice incident to a different vertex of $\beta$. Hence, at least two more flips incident to $\beta$ must take place along the considered path between the two flips incident to $\alpha$: one that removes the loop arc twice incident to $u$ and one that introduces the loop arc twice incident to the other vertex of $\beta$. Indeed, a single flip canot exchange these two loop arcs. Hence at least four flips are incident to $\beta$ along the considered path and the desired inequality follows from Lemma \ref{NOFG.sec.3.lem.1}. 
\end{proof}

Denote by $\zeta$ the arc shared by the triangles of $A_n^-$ incident to $\beta$ and $\gamma$. We next review the case when $\phi$ crosses at least one interior arc of $A_n^-$ other than $\delta$ and $\zeta$. Since $A_n^-$ is built from a zigzag triangulation this may only happen when the triangle of $A_n^-$ that is incident to $\gamma$ is bounded by two interior arcs of $A_n^-$ which are then both crossed by $\phi$. Observe that when $\Sigma$ is one-holed this implies that $n$ is at least $4$ but when $\Sigma$ is not one-holed, this can happen already when $n$ is equal to $3$ (in which case $\gamma$ coincides with $\alpha'$).
\begin{figure}[b]
\begin{centering}
\includegraphics[scale=1]{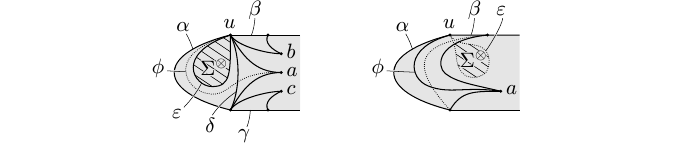}
\caption{The two flips incident to arc $\alpha$ along the path considered in the proof of Proposition \ref{NOFG.sec.3.lem.3}.}\label{NOFG.sec.3.fig.6}
\end{centering}
\end{figure}

\begin{lem}\label{NOFG.sec.3.lem.3}
Assume that $n$ is at least $3$ and consider a geodesic path from $A_n^-$ to $A_n^+$ in $\mathcal{MF}(\Sigma_n)$ along which exactly two flips are incident to $\alpha$. Further consider the arc $\phi$ introduced by the first of these two flips. If $\phi$ crosses $\delta$, $\zeta$, and at least one other interior arc of $A_n^-$, then
$$
d(A_n^-,A_n^+)\geq{d(A_{n-2}^-,A_{n-2}^+)+5}\mbox{.}
$$
where $A_n^\pm$ and $A_{n-2}^\pm$ are built from the same triangulations $X_1^\pm$ and $O_1$.
\end{lem}
\begin{proof}
Denote by $(T_i)_{0\leq{i}\leq{k}}$ the considered path and assume that the first flip incident to $\alpha$ along that path is the one that transforms the triangulation $T_{j-1}$ into the triangulation $T_j$. According to Proposition \ref{NOFG.sec.3.prop.2} the flip that transforms $T_{j-1}$ into $T_j$ must remove $\delta$ and must be as shown on the left of Figure~\ref{NOFG.sec.3.fig.6} where $\phi$ is drawn as a dotted line. In the figure, $a$, $b$, and $c$ denote the vertices of the triangles of $T_j$ incident to $\alpha$, $\beta$, and $\gamma$ that are opposite to these arcs. Note that these three points may not be pairwise distinct.

Now assume that the second flip incident to $\alpha$ along the considered path transforms the triangulation $T_{j'-1}$ into triangulation $T_{j'}$. Since there is no other flip incident to $\alpha$ along $(T_i)_{j\leq{i}<j'}$, this flip must be as shown on the right of Figure~\ref{NOFG.sec.3.fig.6} where both the arc introduced by the flip and the arc $\varepsilon$ are drawn as a dotted lines. One can see that the triangles of $T_j$ and $T_{j'-1}$ incident to $\beta$ cannot be the same because one of the arcs bounding the latter triangle crosses $\varepsilon$ while $T_j$ contains both $\varepsilon$ and the former triangle. Therefore, at least one flip is incident to $\beta$ along $(T_i)_{j\leq{i}<j'}$. One can see on the right of Figure~\ref{NOFG.sec.3.fig.6} that the flip between $T_{j'-1}$ and $T_{j'}$ is also incident to $\beta$. Lemma \ref{NOFG.sec.3.lem.1} then yields
\begin{equation}\label{NOFG.sec.3.eq.2}
d(T_j,A_n^+)\geq{d(T_j\contract\beta,A_n^+\contract\beta)+2}\mbox{.}
\end{equation}

Since $\phi$ crosses both $\delta$ and $\zeta$ (which are the two interior arcs of $A_n^-$ bounding the triangle incident to $\beta$), the triangles of $A_n^-$ and $T_{j-1}$ incident to $\beta$ are necessarily distinct. It then follows from Lemma \ref{NOFG.sec.3.lem.1} that
\begin{equation}\label{NOFG.sec.3.eq.3}
d(A_n^-,T_j)\geq{d(A_n\contract\beta,T_j\contract\beta)+1}\mbox{.}
\end{equation}

Now consider the arc $\eta$ other than $\zeta$ that bounds the triangle incident to $\beta$. By assumption, this arc is crossed by $\phi$. One can see in Figure~\ref{NOFG.sec.3.fig.2} that $\eta$ still bounds the triangle of $A_n^-\contract\beta$ incident to $\gamma$. Moreover, after $\beta$ has been contracted, $\phi$ still crosses $\eta$. As $\phi$ belongs to $T_j\contract\beta$, the triangles of $A_n^-\contract\beta$ and $T_j\contract\beta$ incident to $\gamma$ cannot be the same and according to Lemma \ref{NOFG.sec.3.lem.1},
\begin{equation}\label{NOFG.sec.3.eq.4}
d(A_n^-\contract\beta,T_j\contract\beta)\geq{d(A_n^-\contract\beta\contract\gamma,T_j\contract\beta\contract\gamma)+1}\mbox{.}
\end{equation}

Finally, observe that the triangles of $T_j\contract\beta$ and $A_n^+\contract\beta$ incident to $\gamma$ are different. Indeed, the two interior arcs of $A_n^+\contract\beta$ bounding the latter triangle cross the arc $\varepsilon$ which is contained in $T_j\contract\beta$. Hence, by Lemma \ref{NOFG.sec.3.lem.1},
\begin{equation}\label{NOFG.sec.3.eq.5}
d(T_j\contract\beta,A_n^+\contract\beta)\geq{d(T_j\contract\beta\contract\gamma,A_n^+\contract\beta\contract\gamma)+1}\mbox{.}
\end{equation}

It suffices to combine the inequalities (\ref{NOFG.sec.3.eq.2}), (\ref{NOFG.sec.3.eq.3}), (\ref{NOFG.sec.3.eq.4}), and (\ref{NOFG.sec.3.eq.5}) in order to lower bound the distance of $A_n^-$ and $A_n^+$ in terms of the distance from $T_j\contract\beta\contract\gamma$ to $A_n^-\contract\beta\contract\gamma$ and to $A_n^+\contract\beta\contract\gamma$. Expressing the latter two triangulations with (\ref{NOFG.sec.3.eq.1}) and using the triangle inequality proves the lemma.
\end{proof}

We now examine the case when the only arcs of $A_n^-$ crossed by $\phi$ are $\delta$ and $\zeta$. Consider a triangulation $Y^-_1$ of $\Sigma_1^\otimes$ and the path $(T_i)_{0\leq{i}\leq{l+4}}$ shown in Figure~\ref{NOFG.sec.3.fig.7}, that starts at triangulation $A_n^-$. The first flip in that path removes $\zeta$ and introduces the arc $\zeta'$. The second flip is incident to $\alpha$ and it introduces an arc $\phi$ such that the arcs of $A_n^-$ crossed by $\phi$ are precisely $\delta$ and $\zeta$. The next $l$ flips form a geodesic path between two triangulations of the surface $\Sigma^\otimes_2$ bounded by $\phi$ and $\zeta'$. These two triangulations of $\Sigma^\otimes_2$ each contain a loop arc twice incident to a different vertex of $\phi$ and they respectively surround $X_1^-$ within $T_2$ and $Y^-_1$ within $T_{l+2}$. The last two flips in this path remove $\zeta'$ and then $\phi$. Note that several geodesic paths from $T_2$ to $T_{l+2}$ may be possible. Here, we use a fixed such path. We will need the following straightforward observation.
\begin{figure}
\begin{centering}
\includegraphics[scale=1]{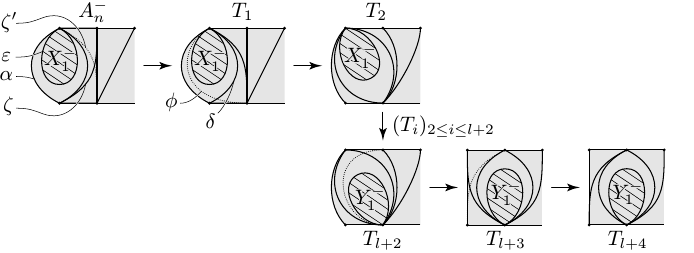}
\caption{The geodesic path in Lemma \ref{NOFG.sec.3.lem.3.5}.}\label{NOFG.sec.3.fig.7}
\end{centering}
\end{figure}

\begin{rem}\label{NOFG.sec.3.rem.1}
Consider some geodesic path in $\mathcal{MF}(\Sigma_n)$ and some arc flipped along it. If the two triangles incident to this arc are not modified by earlier flips, then this arc can be flipped first along the considered geodesic.
\end{rem}

 In order to prove the following, we will repeatedly use Remark \ref{NOFG.sec.3.rem.1}.

\begin{lem}\label{NOFG.sec.3.lem.3.5}
Assume that $n$ is at least $4$ and consider a geodesic from $A_n^-$ to $A_n^+$ in $\mathcal{MF}(\Sigma_n)$ along which exactly two flips are incident to $\alpha$. Let $\phi$ be the arc introduced by the first of these two flips. If
\begin{enumerate}
\item[(i)] $\delta$ and $\zeta$ are the only two arcs of $A_n^-$ crossed by $\phi$ and
\item[(ii)] $\beta$ and $\gamma$ are each incident to at most three flips along the geodesic,
\end{enumerate}
then there exists a triangulation $Y^-_1$ of $\Sigma_1^\otimes$ and a geodesic path $(T_i)_{0\leq{i}\leq{k}}$ from $A_n^-$ to $A_n^+$ in $\mathcal{MF}(\Sigma_n)$ that starts with the $l+2$ flips from Figure~\ref{NOFG.sec.3.fig.7}.
\end{lem}
\begin{proof}
Let $(T_i)_{0\leq{i}\leq{k}}$ denote the considered geodesic path from $A_n^-$ to $A_n^+$ in $\mathcal{MF}(\Sigma_n)$. Assume that the first flip incident to $\alpha$ along this geodesic transforms $T_{j-1}$ into $T_j$. According to Proposition \ref{NOFG.sec.3.prop.2}, this flip must be as shown on the left of Figure~\ref{NOFG.sec.3.fig.6} where $a$, $b$, and $c$ denote the vertices of the triangles of $T_j$ incident to $\alpha$, $\beta$, and $\gamma$. Under the additional assumption that the only two arcs of $A_n^-$ crossed by $\phi$ are $\delta$ and $\zeta$, there is only one possibility for $a$, $c$, and $\phi$: the flip between $T_{j-1}$ and $T_j$ must be as shown on the left of Figure~\ref{NOFG.sec.3.fig.8}. Now observe that the triangles of $A_n^-$, $T_{j-1}$, $T_j$, and $A_n^+$ incident to $\gamma$ are pairwise distinct. Hence, assuming that $\gamma$ is incident to at most three flips along $(T_i)_{0\leq{i}\leq{k}}$, one obtains that there is a single flip incident to $\gamma$ along $(T_i)_{0\leq{i}<j}$, say the flip that transforms $T_{j'-1}$ into $T_{j'}$. As $\delta$ is not modified along $(T_i)_{0\leq{i}<j}$, this flip must replace $\zeta$ with the arc $\zeta'$ as shown in the center of Figure~\ref{NOFG.sec.3.fig.8}. According to Remark \ref{NOFG.sec.3.rem.1}, one can assume that $j'$ is equal to $1$. Now recall that there is no flip incident to either $\alpha$ or $\gamma$ along $(T_i)_{j'\leq{i}<j}$. Hence, the two triangles incident to $\delta$ are not modified along this portion of the path and by Remark~\ref{NOFG.sec.3.rem.1} again, we can assume that $j$ is equal to $2$. In particular, the first two flips along $(T_i)_{0\leq{i}\leq{k}}$ are now precisely the ones shown in Figure~\ref{NOFG.sec.3.fig.7}.

By assumption, there is a unique flip incident to $\alpha$ along $(T_i)_{2\leq{i}\leq{k}}$, say the one between $T_{j''-1}$ and $T_{j''}$. This flip should introduce the triangle of $A_n^+$ incident to $\alpha$ and therefore must be as shown on the right of Figure \ref{NOFG.sec.3.fig.6}. 
\begin{figure}
\begin{centering}
\includegraphics[scale=1]{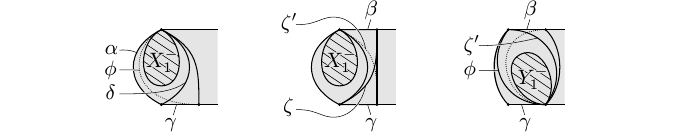}
\caption{Three flips considered in the proof of Lemma \ref{NOFG.sec.3.lem.3.5}. For each of them, the introduced arc is dotted.}\label{NOFG.sec.3.fig.8}
\end{centering}
\end{figure}
Note that this flip is incident to $\beta$ as well and that the triangles of $T_j$ and $T_{j''-1}$ incident to $\beta$ are different. In particular, as least one flip should be incident to $\beta$ along $(T_i)_{j\leq{i}<j''}$. Under the assumption that at most three flips are incident to $\beta$ along $(T_i)_{0\leq{i}\leq{k}}$, it follows that there is exactly one flip incident to $\beta$ along $(T_i)_{j\leq{i}<j''}$ because the first flip along this path and the one that transforms $T_{j''-1}$ into $T_{j''}$ are also incident to $\beta$. Say the unique flip incident to $\beta$ along $(T_i)_{j\leq{i}<j''}$ changes $T_{j'''-1}$ into $T_{j'''}$. This flip must be as shown on the right of Figure~\ref{NOFG.sec.3.fig.8}. Observe in particular that in $T_{j'''-1}$, the arcs $\phi$ and $\zeta'$ surround a triangulation of $\Sigma_2^\otimes$ that contains a loop arc bounding a certain triangulation $Y_1^-$ of $\Sigma_1^\otimes$. Moreover the arcs $\phi$ and $\zeta'$ are not modified along $(T_i)_{j\leq{i}<j'''}$. Therefore, the flips that have been performed in the subsurface $\Sigma^\otimes_2$ surrounded by $\phi$ and $\zeta'$ form a geodesic path from the triangulation of this subsurface contained in $T_j$ to the one contained in $T_{j'''-1}$. By Remark \ref{NOFG.sec.3.rem.1}, it can be assumed that these flips are the $l$ flips that take place after the first two along $(T_i)_{0\leq{i}\leq{k}}$. Also by this remark, we can assume that $j'''$ is equal to $l+3$ and that $j''$ is equal to $l+4$ which completes the proof.
\end{proof}

We prove the following as a consequence of Lemma \ref{NOFG.sec.3.lem.3.5}.

\begin{lem}\label{NOFG.sec.3.lem.4}
Assume that $n$ is at least $3$ and consider a geodesic path from $A_n^-$ to $A_n^+$ in $\mathcal{MF}(\Sigma_n)$ along which exactly two flips are incident to $\alpha$. Further consider the arc $\phi$ introduced by the first of these two flips. If $\delta$ and $\zeta$ are the only two interior arcs of $A_n^-$ crossed by $\phi$, then either
\begin{equation}\label{NOFG.sec.3.lem.4.eq.0}
d(A_n^-,A_n^+)\geq{d(A_{n-2}^-,A_{n-2}^+)+5}
\end{equation}
where $A_n^\pm$ and $A_{n-2}^\pm$ are built from the same triangulations $X_1^\pm$ and $O_1$ or $n$ is at least $4$ and there exists a triangulation $Y_1^-$ of $\Sigma_1^\otimes$ such that
\begin{equation}\label{NOFG.sec.3.lem.4.eq.0.1}
d(A_n^-,A_n^+)\geq{d(A_{n-3}^-,A_{n-3}^+)+d(X_1^-,Y_1^-)+8}
\end{equation}
where $A_n^\pm$ and $A_{n-3}^+$ are built from the same triangulations $X_1^\pm$ and $O_1$ while $A_{n-3}^-$ is built from $Y_1^-$ and $O_1$ instead of $X_1^-$ and $O_1$.
\end{lem}
\begin{proof}
Assume that $\phi$ crosses $\delta$ and $\zeta$ but no other interior arc of $A_n^-$. First consider the case when $\beta$ is incident to at least four flips along the considered geodesic path. In that case it follows from Lemma \ref{NOFG.sec.3.lem.1} that
$$
d(A_n^-,A_n^+)\geq{d(A_n^-\contract\beta,A_n^+\contract\beta)+4}
$$
and since the triangles of $A_n^-\contract\beta$ and $A_n^+\contract\beta$ incident to $\gamma$ are different, the same lemma and Equation (\ref{NOFG.sec.3.eq.1}) prove (\ref{NOFG.sec.3.lem.4.eq.0}) where $A_n^\pm$ and $A_{n-2}^\pm$ are built from the same triangulations $X_1^\pm$ and $O_1$. Now observe that
$$
A_n^\pm\contract\beta\contract\gamma=A_n^\pm\contract\gamma\contract\beta\mbox{.}
$$

As the triangles of $A_n^-\contract\gamma$ and $A_n^+\contract\gamma$ incident to $\beta$ are different, the same argument also proves (\ref{NOFG.sec.3.lem.4.eq.0}) when $\gamma$ is incident to at least four flips along the considered geodesic path. Now assume that $\beta$ and $\gamma$ are both incident to at most three flips along the geodesic. In that case, by Lemma~\ref{NOFG.sec.3.lem.3.5}, $n$ is at least $4$ and there exists a triangulation $Y_1^-$ of $\Sigma_1^\otimes$ and a geodesic path $(T_i)_{0\leq{i}\leq{k}}$ that starts with the $l+4$ flips shown in Figure \ref{NOFG.sec.3.fig.7}. Now observe that
$$
T_{l+4}\contract\gamma\contract\beta\contract\alpha=A_{n-3}^-
$$
where $A_{n-3}^-$ is built from $Y_1^-$ and $O_1$ instead of $X_1^-$ and $O_1$. Moreover,
$$
A_n^+\contract\gamma\contract\beta\contract\alpha=A_{n-3}^+
$$
where $A_n^+$ and $A_{n-3}^+$ are both built from $X_1^+$. However, the triangles incident to $\alpha$ in $T_{l+4}\contract\gamma\contract\beta$ and $A_n^+\contract\gamma\contract\beta$ are distinct. Indeed, the former triangle is bounded by a loop arc while the latter is not. Therefore,
$$
d(T_{l+4},A_n^+)\geq{d(A_{n-3}^-,A_{n-3}^+)+1}
$$
and since $T_{l+4}$ is reached after $l+4$ flips along a geodesic from $A_n^-$ to $A_n^+$,
\begin{equation}\label{NOFG.sec.3.lem.4.eq.1}
d(A_n^-,A_n^+)\geq{d(A_{n-3}^-,A_{n-3}^+)+l+5}
\end{equation}
where $A_n^\pm$ and $A_{n-3}^+$ are built from the same triangulations $X_1^\pm$ and $O_1$ while $A_{n-3}^-$ is built from $Y_1^-$ and $O_1$. Finally, by Proposition \ref{NOFG.sec.3.prop.2.5},
$$
l\geq{d(X_1^-,Y_1^-)+3}
$$
and combining this with (\ref{NOFG.sec.3.lem.4.eq.1}) completes the proof.
\end{proof}

We are now ready to prove Theorem \ref{NOFG.sec.3.thm.1}.

\begin{proof}[Proof of Theorem \ref{NOFG.sec.3.thm.1}]
If $\Sigma$ is one holed, we choose $X_1^-$ and $X_1^+$ such that
\begin{equation}\label{NOFG.sec.3.thm.1.eq.0}
d(X_1^-,X_1^+)=\mathrm{diam}(\mathcal{MF}(\Sigma_1))\mbox{.}
\end{equation}

If $\Sigma$ is not one-holed, we take any two triangulations of $\Sigma^\otimes_1$ for $X_1^-$ and $X_1^+$. It suffices to show that the distance of $A_n^-$ and $A_n^+$ in $\mathcal{MF}(\Sigma_n)$ satisfies
$$
d(A_n^-,A_n^+)\geq\biggl\lfloor\frac{5}{2}n\biggr\rfloor-2+d(X_1^-,X_1^+)
$$
when $\Sigma$ is one-holed and that it is at least the right-hand side of (\ref{NOFG.sec.3.eq.0}) when it is not. We proceed by induction on $n$. If $n$ is equal to $1$, the result is immediate when $\Sigma$ is one-holed because $A_1^-$ and $A_1^+$ coincide with $X_1^-$ and $X_1^+$. If $\Sigma$ is not one-holed, it is also immediate because the right-hand side of (\ref{NOFG.sec.3.eq.0}) vanishes. If $n$ is equal to $2$, the result follows from Proposition \ref{NOFG.sec.3.prop.2.5}.

Now assume that $n$ is at least $3$ and let us examine the number of flips incident to $\alpha$ along the geodesic path $(T_i)_{0\leq{i}\leq{k}}$ considered above. If there are three such flips, then one obtains the desired inequality by induction from Lemma \ref{NOFG.sec.3.lem.0}. If there are at most two flips incident to $\alpha$ along $(T_i)_{0\leq{i}\leq{k}}$, then according to Proposition~\ref{NOFG.sec.3.prop.2} there should be exactly two such flips and we consider the arc $\phi$ introduced by the first of these flips. If $\phi$ crosses $\delta$ and no other interior arc of $A_n^-$ or it crosses $\delta$, $\zeta$ and at least one other interior arc of $A_n^-$, then the desired inequality follows by induction from Lemmas~\ref{NOFG.sec.3.lem.2} and~\ref{NOFG.sec.3.lem.3}.

Finally, assume that the only two arcs of $A_n^-$ crossed by $\phi$ are $\delta$ and $\zeta$. According to Lemma \ref{NOFG.sec.3.lem.4}, there are just two possibilities. The first possibility is that (\ref{NOFG.sec.3.lem.4.eq.0}) hold and the result follows by induction from that inequality. The other possibility is that $n$ is at least $4$ and 
\begin{equation}\label{NOFG.sec.3.thm.1.eq.3}
d(A_n^-,A_n^+)\geq{d(A_{n-3}^-,A_{n-3}^+)+d(X_1^-,Y_1^-)+8}
\end{equation}
for some triangulation $Y_1^-$ of $\Sigma^\otimes$, where $A_n^\pm$ and $A_{n-3}^+$ are built from the same triangulations $X_1^\pm$ and $O_1$ while $A_{n-3}^-$ is built from $Y_1^-$ and $O_1$ instead of $X_1^-$ and $O_1$. Note in particular, that at the moment, the theorem holds when $n$ is equal to $3$. Hence, if $\Sigma$ is not one-holed, it sufficees to lower bound the distance of $X_1^-$ and $Y_1^-$ by $0$ and the result follows by induction from (\ref{NOFG.sec.3.thm.1.eq.3}).

If $\Sigma$ is one-holed, then by induction,
\begin{equation}\label{NOFG.sec.3.thm.1.eq.4}
d(A_{n-3}^-,A_{n-3}^+)\geq\biggl\lfloor\frac{5}{2}n-\frac{15}{2}\biggr\rfloor-2+d(Y_1^-,X_1^+)
\end{equation}
because $A_{n-3}^-$ is built from $Y_1^-$ instead of $X_1^-$. By the triangle inequality,
$$
d(X_1^-,Y_1^-)+d(Y_1^-,X_1^+)\geq{d(X_1^-,X_1^+)}\mbox{.}
$$

Combining this with (\ref{NOFG.sec.3.thm.1.eq.0}), (\ref{NOFG.sec.3.thm.1.eq.3}), and (\ref{NOFG.sec.3.thm.1.eq.4}) provides the desired inequality.
\end{proof}

\section{An upper bound on the diameter of $\mathcal{F}(\mathrm{M}_n)$}\label{NOFG.sec.4}

Recall that $\mathrm{M}$ denotes the M{\"o}bius strip without punctures. By Theorem \ref{NOFG.sec.3.thm.1},
$$
\mathrm{diam}(\mathcal{F}(\mathrm{M}_n))\geq\biggl\lfloor\frac{5}{2}n\biggr\rfloor-2\mbox{.}
$$

In this section, we shall prove that this bound is sharp up to an additive constant, which implies that the lower bound on $c_\Sigma$ stated by Theorem \ref{NOFG.sec.1.thm.1} is sharp when $\Sigma$ is a non-orientable filling surface. We will use the representation of $\mathrm{M}$ as a disk with an inserted cross-cap as on the left of Figure~\ref{NOFG.sec.1.fig.1}.

We collect some observations about the arcs bounding a triangle of a triangulation of $\mathrm{M}_n$. Recall that $\mathcal{P}$ is the set of the marked points placed in the boundary of $\mathrm{M}$ in order to build $\mathrm{M}_n$ and that these points serve as the vertices of the triangulations of $\mathrm{M}_n$. In the following statement, a non-boundary arc $\alpha$ between two points from $\mathcal{P}$ is \emph{non-separating} when $\mathrm{M}_n\mathord{\setminus}\alpha$ is connected and \emph{separating} otherwise. In the cross-cap model of the M{\"o}bius strip, $\alpha$ is non-separating precisely when it goes through the cross-cap.

\begin{prop}\label{NOFG.sec.4.prop.1}
If $t$ is a triangle of a triangulation $T$ of $\mathrm{M}_n$, then
\begin{enumerate}
\item[(i)] $t$ is bounded by at least two arcs of $T$,
\item[(ii)] if $t$ is bounded by exactly two arcs of $T$ then both arcs are loops and exactly one of these loops is non-separating, and
\item[(iii)] if $t$ is bounded by three different arcs of $T$ then either exactly two of these arcs or none of them are non-separating.
\end{enumerate}
\end{prop}

The situation described by the second assertion in the statement of Proposition~\ref{NOFG.sec.4.prop.1} is illustrated on the left of Figure~\ref{NOFG.sec.1.fig.1} where the triangle $t$ is the one bounded by the two loop arcs. Note that the non-separating loop bounding $t$ serves as two edges of $t$. Consider distinct marked points $u$ and $v$ in $\mathcal{P}$. Cutting the boundary of $\mathrm{M}$ at $u$ and $v$ splits it into two arcs which we will denote by $\beta_{u,v}$ and $\beta_{v,u}$ with the convention that $\beta_{u,v}$ is the portion of the boundary of $\mathrm{M}$ that lies clockwise from $u$ and counterclockwise from $v$. We call the \emph{clockwise distance from $u$ to $v$} the number of boundary arcs of $\mathrm{M}_n$ contained in $\beta_{u,v}$ or equivalently the number of points in $\mathcal{P}\cap\beta_{u,v}$ minus one. This quantity will be denoted by $d^-(u,v)$. Likewise, the \emph{counter-clockwise distance} $d^+(u,v)$ from $u$ to $v$ will be the number of boundary arcs of $\mathrm{M}_n$ contained in $\beta_{v,u}$.

Now assume that $u$ and $v$ are the two extremities of an arc $\alpha$. If $\alpha$ is a non-separating arc, then its length will be defined as
$$
\ell(\alpha)=\min\bigl\{d^-(u,v),d^+(u,v)\bigr\}\mbox{.}
$$

Note that in this definition, the right-hand side remains the same if one exchanges $u$ and $v$ because $d^-(u,v)$ coincides with $d^+(v,u)$ so $\ell(\alpha)$ only depends on $\alpha$ itself and not on the labeling of its vertices. This defines the length of a non-separating arc $\alpha$ except if $\alpha$ is a loop arc and we will naturally use the convention that a non-separating loop arc has length $0$.

If $u$ and $v$ are the two extremities of a separating arc $\alpha$, the length $\ell(\alpha)$ of that arc is defined in a different way. In that case, $\mathrm{M}_n\mathord{\setminus}\alpha$ has two connected components, and one of these connected components is a disk bounded by $\alpha$ and either $\beta_{u,v}$ or $\beta_{v,u}$. If that disk is bounded by $\beta_{u,v}$, we set
$$
\ell(\alpha)=d^-(u,v)
$$
and if it is bounded by $\beta_{v,u}$, we set
$$
\ell(\alpha)=d^+(u,v)\mbox{.}
$$

Finally, if $\alpha$ is a boundary arc of $\mathrm{M}_n$, we set $\ell(\alpha)$ to $1$.

\begin{defn}
Consider a triangulation $T$ of $\mathrm{M}_n$ and a triangle $t$ of $T$. We say that $t$ is a \emph{central triangle} of $T$ when it is bounded by at least one non-separating arc and the lengths of the arcs bounding $t$ sum to $n$.
\end{defn}

Central triangles will appear in our construction of short paths between pairs of triangulation of $\mathrm{M}_n$. In a first step, we shall prove the following.

\begin{thm}\label{NOFG.sec.4.thm.1}
A triangulation of $\mathrm{M}_n$ has at least one central triangle.
\end{thm}

The proof of Theorem \ref{NOFG.sec.4.thm.1} is split into three lemmas. Recall that any triangulation of $\mathrm{M}_n$ contains at least one non-separating arc.

\begin{lem}\label{NOFG.sec.4.lem.1}
Consider a triangulation $T$ of $\mathrm{M}_n$. If a vertex $u$ of $T$ is incident to a non-separating loop $\alpha$ and no other non-separating arc of $T$, then
\begin{itemize}
\item[(i)] $\alpha$ is the only non-separating arc contained in $T$ and
\item[(ii)]  $T$ has exactly one central triangle.
\end{itemize}
\end{lem}
\begin{proof}
If $u$ is not incident to a non-separating arc of $T$ other than $\alpha$, then $\alpha$ serves as two edges of a triangle $t$ which is further bounded by a separating loop $\beta$ (see the left of Figure~\ref{NOFG.sec.1.fig.1}). As $\beta$ surrounds the topology of $\mathrm{M}_n$, there cannot be a non-separating arc in $T$ other than $\alpha$. By definition, the length of $\alpha$ is $0$ and the length of $\beta$ is $n$. Hence, $t$ is a central triangle of $T$ and since $\alpha$ is the only non-separating arc in $T$, this is the only central triangle of $T$.
\end{proof}

Now consider a triangulation $T$ of $\mathrm{M}_n$. For every vertex $u$ of $T$ incident to a non-separating arc in $T$ that is not a loop let $\mathcal{V}_T(u)$ denote the set of the vertices $v$ of $T$ different from $u$ such that there is a non-separating arc in $T$ with vertices $u$ and $v$. We can prove the following statement.

\begin{lem}\label{NOFG.sec.4.lem.2}
Consider a triangulation $T$ of $\mathrm{M}_n$ and a vertex $u$ of $T$ incident to at least one non-separating arc in $T$ that is not a loop. If
$$
\min\bigl\{d^-(u,v):v\in\mathcal{V}_T(u)\bigr\}\leq\frac{n}{2}\leq\max\bigl\{d^-(u,v):v\in\mathcal{V}_T(u)\bigr\}\mbox{,}
$$
then $T$ has at least one central triangle.
\end{lem}
\begin{proof}
First assume that a point $v$ in $\mathcal{V}_T(u)$ satisfies
$$
d^-(u,v)=\frac{n}{2}\mbox{.}
$$

In that case any triangle $t$ of $T$ incident to the non-separating arc $\alpha$ between $u$ and $v$ is a central triangle. Indeed, by the triangle inequality, 
$$
\ell(\beta)+\ell(\gamma)\geq\ell(\alpha)
$$
where $\beta$ and $\gamma$ are the two arcs that bound $t$ other from $\alpha$. This implies
$$
\ell(\alpha)+\ell(\beta)+\ell(\gamma)\geq{n}
$$
and since the lengths of the arcs bounding a triangle sum to at most $n$, this shows that $t$ is a central triangle. Now assume that
$$
\min\bigl\{d^-(u,v):v\in\mathcal{V}_T(u)\bigr\}<\frac{n}{2}<\max\bigl\{d^-(u,v):v\in\mathcal{V}_T(u)\bigr\}
$$
and consider the point $a$ in $\mathcal{V}(u)$ such that $d^-(u,a)$ is the largest possible under the constraint that $d^-(u,a)$ is at most $n/2$. Likewise, let $b$ be the point in $\mathcal{V}(u)$ such that $d^-(u,b)$ is greater than $n/2$ and $d^-(u,b)$ is the smallest possible under that requirement. By construction, $u$, $a$, and $b$ are the three distinct vertices of a triangle $t$ of $T$. The two arcs $\alpha$ and $\beta$ incident to $u$ bounding this triangle are non-separating. The third arc bounding this triangle---let us denote it by $\gamma$---is separating according to Proposition \ref{NOFG.sec.4.prop.1} and by construction,
\begin{equation}\label{NOFG.sec.4.lem.1.eq.2}
\ell(\gamma)=d^-(u,b)-d^-(u,a)\mbox{.}
\end{equation}

However, as $d^-(u,a)$ is at most $n/2$ and $d^-(u,b)$ it greater than $n/2$,
$$
d^-(u,b)-d^-(u,a)=n-\ell(\alpha)-\ell(\beta)
$$
and combining this with (\ref{NOFG.sec.4.lem.1.eq.2}) shows that $t$ is a central triangle of $T$.
\end{proof}

Consider a triangulation $T$ of $\mathrm{M}_n$ that contains at least one non-separating non-loop arc. According to the following lemma, the condition in the statement of Lemma \ref{NOFG.sec.4.lem.2} is satisfied by at least one vertex $u$ of $T$ .

\begin{lem}\label{NOFG.sec.4.lem.3}
Consider a triangulation $T$ of $\mathrm{M}_n$ containing at least one non-separating arc that is not a loop. If every vertex $u$ of $T$ incident to at least one non-separating non-loop arc of $T$ satisfies either
\begin{equation}\label{NOFG.sec.4.lem.3.eq.2}
\max\bigl\{d^-(u,v):v\in\mathcal{V}_T(u)\bigr\}<\frac{n}{2}
\end{equation}
or
\begin{equation}\label{NOFG.sec.4.lem.3.eq.3}
\min\bigl\{d^-(u,v):v\in\mathcal{V}_T(u)\bigr\}>\frac{n}{2}\mbox{,}
\end{equation}
then $T$ has at least one central triangle.
\end{lem}
\begin{proof}
Let $\mathcal{Q}$ be the set of the points in $\mathcal{P}$ incident to at least one non-separating arc in $T$. By Lemma \ref{NOFG.sec.4.lem.1}, every point in $\mathcal{Q}$ is incident to at least one non-separating arc in $T$ that is not a loop. Assume that every point $u$ contained in $\mathcal{Q}$ satisfies either (\ref{NOFG.sec.4.lem.3.eq.2}) or (\ref{NOFG.sec.4.lem.3.eq.3}). By this assumption,
\begin{equation}\label{NOFG.sec.4.lem.3.eq.4}
\mathcal{Q}=\mathcal{Q}^-\cup\mathcal{Q}^+
\end{equation}
where $\mathcal{Q}^-$ is the set of the points in $\mathcal{Q}$ that satisfy (\ref{NOFG.sec.4.lem.3.eq.2}) and $\mathcal{Q}^+$ the set of the points in $\mathcal{Q}$ that satisfy (\ref{NOFG.sec.4.lem.3.eq.3}). In particular, $\mathcal{Q}^-$ and $\mathcal{Q}^+$ cannot both be empty. It turns out that they are both non-empty. Indeed, by symmetry we can assume without loss of generality that $\mathcal{Q}^-$ is non-empty and pick a point $u$ in $\mathcal{Q}^-$. Note that $\mathcal{V}_T(u)$ is a subset of $\mathcal{Q}$. In fact, $\mathcal{V}_T(u)$ is a subset of $\mathcal{Q}^+$ because
$$
d^-(u,w)=n-d^-(w,u)
$$
for every point $w$ in $\mathcal{V}_T(u)$, which proves that $\mathcal{Q}^+$ is non-empty as well. In particular, $\mathcal{Q}^-$ and $\mathcal{Q}^+$ form a partition of $\mathcal{Q}$.
 
Denote $q$ the number of points in $\mathcal{Q}$ and observe that the set of all arcs in $T$ whose two endpoints belong to $\mathcal{Q}$ form a triangulation $U$ of $\mathrm{M}_q$. 
\begin{figure}[b]
\begin{centering}
\includegraphics[scale=1]{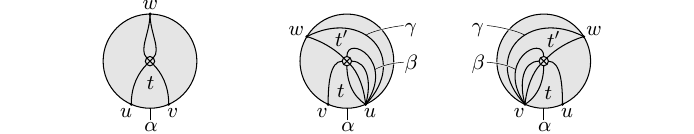}
\caption{The triangles $t$ and $t'$ from the proof of Lemma \ref{NOFG.sec.4.lem.3} represented within the surface $\mathrm{M}_q$.}\label{NOFG.sec.4.fig.1}
\end{centering}
\end{figure}
Equivalently, $U$ is obtained from $T$ by cutting away the triangles that are not incident to any non-separating arc. Since $\mathcal{Q}^-$ and $\mathcal{Q}^+$ are both non-empty, one can find a boundary arc $\alpha$ of $\mathrm{M}_q$ with one end in $\mathcal{Q}^-$ and the other in $\mathcal{Q}^+$. Let $u$ and $v$ denote the vertices of $\alpha$ labeled in such a way that $u$ belongs to $\mathcal{Q}^-$ and $v$ to $\mathcal{Q}^+$. Further denote by $t$ the triangle of $U$ incident to $\alpha$.
 
Let us show that $t$ cannot have three distinct vertices. Indeed, otherwise, $t$ must be as shown on the left of Figure~\ref{NOFG.sec.4.fig.1} where $w$ is the vertex of $t$ that does not belong to $\alpha$. In this situation $\alpha$ must lie counter-clockwise from $u$ and clockwise from $v$ because $u$ belongs to $\mathcal{Q}^-$ and $v$ to $\mathcal{Q}^+$. In particular,
$$
d^-(u,w)<\frac{n}{2}<d^-(v,w)\mbox{.}
$$
 
This implies that $d^-(w,u)$ is greater than $n/2$ and $d^-(w,v)$ is less than $n/2$. Hence, $w$ belongs to $\mathcal{Q}\mathord{\setminus}(\mathcal{Q}^-\cup\mathcal{Q}^+)$, which contradicts (\ref{NOFG.sec.4.lem.3.eq.4}).
 
We have just proved that $u$ and $v$ are the only two vertices of $t$. As a consequence, one of the arcs of $T$ that bounds $t$ is a non-separating arc between $u$ and $v$ and another is a loop $\beta$ twice incident to $u$ or to $v$ as shown at the center or on the right of Figure~\ref{NOFG.sec.4.fig.1} depending on what point $\beta$ is twice incident to. Note that in both cases, $\alpha$ must lie clockwise from $u$ and counter-clockwise from $v$ because $u$ belongs to $\mathcal{Q}^-$, $v$ to $\mathcal{Q}^+$, and these points are the two endpoints of a non-separating arc. Now consider the triangle $t'$ incident to $\beta$ and different from $t$ which is also shown in Figure~\ref{NOFG.sec.4.fig.1}. Denote by $w$ the vertex of $t'$ that does not belong to $\alpha$ and by $\gamma$ the separating arc that bounds $t'$.

If $\beta$ is twice incident to $u$, then $d^-(u,w)$ is less than $n/2$ because $u$ belongs to $\mathcal{Q}^-$ and there is a non-separating non-loop arc of $T$ between $u$ and $w$. As a consequence, $\ell(\gamma)$ is greater than $n/2$. As $t'$ is bounded by $\gamma$ and two non-separating arcs of $T$, it follows that the lengths of these three arcs sum to $n$ and therefore $t'$ is a central triangle of $T$. A similar argument shows that $t'$ is also a central triangle when $\beta$ is twice incident to $v$.
\end{proof}

Theorem \ref{NOFG.sec.4.thm.1} follows immediately from Lemmas~\ref{NOFG.sec.4.lem.1}, \ref{NOFG.sec.4.lem.2}, and~\ref{NOFG.sec.4.lem.3}. We now build explicit paths in $\mathcal{F}(\mathrm{M}_n)$ between any two triangulations, whose length will allow to upper bound the diameter of that graph. These paths will be through the family of triangilations $C_{u}(v,w)$ represented on the left of Figure~\ref{NOFG.sec.4.fig.2}. In these triangulations, $v$ and $w$ are two (non-necessarily distinct) points in $\mathcal{P}\mathord{\setminus}\{u\}$ such that $v$ belongs to $\beta_{u,w}$. They contain a non-separating loop twice incident to $u$ and the two triangles incident to that loop admit $v$ and $w$ as their only other vertex. 
\begin{figure}[b]
\begin{centering}
\includegraphics[scale=1]{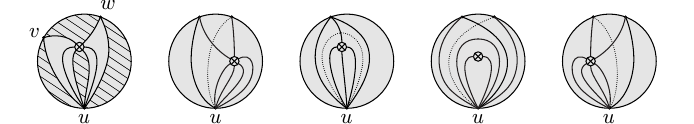}
\caption{The triangulation $C_u(v,w)$ and four flips in this triangulation, where the introduced arcs are dotted.}\label{NOFG.sec.4.fig.2}
\end{centering}
\end{figure}
All the other arcs in $C_u(v,w)$ are incident to 
$u$ and to a point different from $u$ (so in particular, $C_u(v,w)$ does not contain a separating loop). Assuming that $C_u(v,w)$ and $C_u(v',w')$ are two such triangulations, it is not hard to see that their distance in $\mathcal{MF}(\mathrm{M}_n)$ satisfies
\begin{equation}\label{NOFG.sec.4.eq.1}
d\bigl(C_u(v,w),C_u(v',w')\bigr)\leq\max\bigl\{d^-(u,w),d^-(u,w')\bigr\}+1\mbox{.}
\end{equation}

Indeed, assume without loss of generality that $d^-(u,v)$ is at most $d^-(u,v')$ and consider the four flips shown in Figure~\ref{NOFG.sec.4.fig.2}. One can move $v$ clockwise using the first flip on the left until either it coincides with $v'$ or with $w'$. In the latter case, one can perform the second flip that introduces a separating loop and move the triangle incident to that loop clockwise using the third flip until the vertex of that triangle distinct from $u$ is equal to $v'$. After removing the separating loop by a flip, $v$ and $v'$ coincide. Finally one can use the last flip on the right of Figure~\ref{NOFG.sec.4.fig.2} to move $w$ or $w'$ until they coincide. In this process, the number of flips that have been performed is at most
$$
\max\bigl\{d^-(v,w),d^-(v,w')\bigr\}
$$
to move vertices clockwise plus possibly two flips that introduce or remove a separating loop. As $v$ is not equal to $u$, then
$$
\max\bigl\{d^-(v,w),d^-(v,w')\bigr\}=\max\bigl\{d^-(u,w),d^-(u,w')\bigr\}-1
$$
and this proves (\ref{NOFG.sec.4.eq.1}). In order to build the rest of our paths, we exhibit non-separating arcs with certain properties in the triangulations of $\mathrm{M}_n$.

\begin{lem}\label{NOFG.sec.4.lem.6}
Consider a triangulation $T$ of $\mathrm{M}_n$ and a point $u$ in $\mathcal{P}$. If $T$ does not contain a non-separating arc incident to $u$, then it contains a non-separating arc incident to a point $v$ different from $u$ and whose other endpoint belongs to $\beta_{u,v}$ when $d^-(u,v)$ is at most $n/2$ or to $\beta_{v,u}$ otherwise.
\end{lem}
\begin{proof}
Assume that $u$ is not incident to a non-separating arc in $T$. In that case there exists a separating arc $\beta$ in $T$ that is not incident to $u$ and such that, splitting $T$ along $\beta$ results in a triangulation of $\mathrm{M}_{n-\ell(\beta)+1}$ and a triangulation of a disk. Pick for $\beta$ the longest such separating arc. 
\begin{figure}[b]
\begin{centering}
\includegraphics[scale=1]{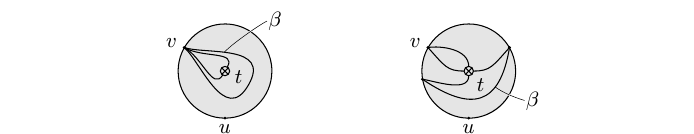}
\caption{The arc $\beta$ and the triangle $t$.}\label{NOFG.sec.4.fig.4}
\end{centering}
\end{figure}
It follows from this choice that the triangle $t$ of $T$ contained in $\mathrm{M}_{n-\ell(\beta)+1}$ and incident to $\beta$ is as shown in Figure~\ref{NOFG.sec.4.fig.4} depending on whether $\beta$ is a loop or not.

If $\beta$ is a loop then let $v$ be the point $\beta$ is twice incident to. Since $\beta$ is a separating loop, $T$ also contains a non-separating loop twice incident to $v$ and the result is immediate because $v$ belongs to both $\beta_{u,v}$ and $\beta_{v,u}$.

If $\beta$ is not a loop then $t$ is bounded by two distinct non-separating arcs $\gamma^-$ and $\gamma^+$ (one of which may be a loop) as shown on the right of Figure~\ref{NOFG.sec.4.fig.4}. Let $v$ be the unique vertex shared by these two non-separating arcs. Observe that one of these arcs, say $\gamma^-$ has its two endpoints in $\beta_{u,v}$ and the other, say $\gamma^+$ has its two endpoints in $\beta_{v,u}$. Picking $\gamma^-$ or $\gamma^+$ depending on whether $d^-(u,v)$ is at most $n/2$ or greater than $n/2$ proves the lemma.
\end{proof}

We are ready to upper bound the diameter of $\mathcal{F}(\mathrm{M}_n)$.

\begin{thm}\label{NOFG.sec.4.thm.2}
For every positive $n$,
$$
\mathrm{diam}(\mathcal{F}(\mathrm{M}_n))\leq\biggl\lfloor\frac{5}{2}n\biggr\rfloor\mbox{.}
$$
\end{thm}
\begin{proof}
According to Theorem~\ref{NOFG.sec.4.thm.1}, $T^-$ has at least one central triangle $t$. This triangle may be bounded by just two arcs, in which case these two arcs are a separating loop and a non-separating loop twice incident to the same vertex $u$ as shown on the left of Figure~\ref{NOFG.sec.1.fig.1}. If however, $t$ is bounded by three distinct arcs it follows from Proposition~\ref{NOFG.sec.4.prop.1} that two of these arcs are non-separating and share a unique vertex that we shall denote by $u$.

Now that $u$ is set, let us consider $T^+$ for a moment. If this triangulation contains a separating loop twice incident to $u$ then we denote by $\gamma$ the non-separating loop contained in $T^+$. If however, $T^+$ does not contain a separating arc twice incident to $u$ then either this triangulation contains a non-separating arc $\gamma$ between $u$ and a vertex $w$ different from $u$ or it does not contain any non-separating arc incident to $u$. In the latter case, it follows from Lemma~\ref{NOFG.sec.4.lem.6} that $T^+$ contains a non-separating arc $\gamma$ whose extremities are a point $w$ in $\mathcal{P}\mathord{\setminus}\{u\}$ and a point $v$ contained in $\beta_{u,w}\mathord{\setminus}\{u\}$ when $d^-(u,w)$ is at most $n/2$ and in $\beta_{w,u}\mathord{\setminus}\{u\}$ otherwise. We can assume that $d^-(u,w)$ is at most $n/2$ by reversing the clockwise order around the boundary of $\mathrm{M}_{n}$ if needed. In particular, this orientation-reversing operation can be performed without loss of generality as it preserves the centrality of $t$ as a triangle of $T^-$. 
\begin{figure}
\begin{centering}
\includegraphics[scale=1]{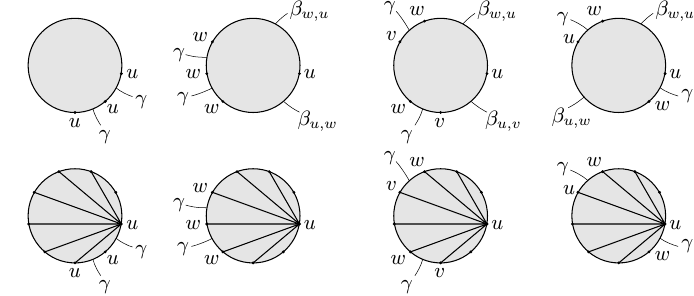}
\caption{$\Delta_{n+2}$ (top) and $F^+$ (bottom).}\label{NOFG.sec.4.fig.3}
\end{centering}
\end{figure}

Cutting $T^+$ along $\gamma$ results in a triangulation $U^+$ of a disk $\Delta_{n+2}$ shown in the top row of Figure~\ref{NOFG.sec.4.fig.3}. The case when $T^+$ contains a separating loop twice incident to $u$ is shown on the left where the separating loop is the only represented interior arc. The cases when $\gamma$ is a loop twice incident to $w$, an arc incident to $u$ and $w$, or and arc incident to $v$ and $w$ are further shown in the same row from left to right. One can see that in these three cases, $\beta_{w,u}$ is a connected subset of the boundary of $\Delta_{n+2}$.

As is well known~\cite{SleatorTarjanThurston1988}, one can transform $U^+$ in a triangulation $F^+$ of $\Delta_{n+2}$ whose all interior arcs are incident to the same vertex using a sequence of at most $n-k-1$ flips where $k$ denotes the number of interior arcs of $U$ that are already incident to that vertex. We pick for that vertex the copy of $u$ in the boundary of $\Delta_{n+2}$ shown in the bottom row of Figure~\ref{NOFG.sec.4.fig.3} in each of the four considered cases. We perform this sequence of at most $n-1$ flip within $T^+$ and then flip $\gamma$ if that arc is not incident to $u$. Note that if $T^+$ contains a separating loop twice incident to $u$, then at most $n-2$ flips are required between $U^+$ and $F^+$ because that loop already belongs to $F^+$ but in that case we further flip this loop in order to introduce a non-separating arc incident to $u$. After this sequence of at most $n$ flips, we have reached the triangulation $C_u(v,w)$ where, in the case when $v$ or $w$ has not been defined yet these points are equal to the point of $\mathcal{P}$ that immediately follows $u$ clockwise. Therefore,
\begin{equation}\label{NOFG.sec.4.thm.2.eq.1}
d\bigl(T^+,C_u(v,w)\bigr)\leq{n}\mbox{.}
\end{equation}

Let us turn our attention back to $T^-$. Since $t$ is a central triangle of $T^-$, one of the non-separating arcs $\gamma'$ that bound $t$ is either a loop twice incident to $u$ or an arc between $u$ and a point $w'$ different from $u$ but such that $d^-(u,w')$ is at most $n/2$. Cutting $T^-$ along $\gamma'$ and considering a sequence of flips like the one we used to transform $T^+$ into $C_u(v,w)$ except that we do not flip $\gamma'$ (because it is already incident to $u$), we can change $T^-$ into a triangulation $C_u(v',w')$ where $v'$ is the point of $\mathcal{P}$ that immediately follows $u$ clockwise. Note that $\gamma'$ does not need to be flipped in that case because it is already incident to $u$ and therefore the distance of $T^-$ and $C_u(v',w')$ in $\mathcal{MF}(\mathrm{M}_n)$ satisfies
\begin{equation}\label{NOFG.sec.4.thm.2.eq.2}
d\bigl(T^-,C_u(v',w')\bigr)\leq{n-1}\mbox{.}
\end{equation}

Since both $d^-(u,w)$ and $d^-(u,w')$ are at most $n/2$, Combining (\ref{NOFG.sec.4.thm.2.eq.1}) and (\ref{NOFG.sec.4.thm.2.eq.2}) with (\ref{NOFG.sec.4.eq.1}) provides the desired inequality.
\end{proof}

\section{Bounds on the diameter of $\mathcal{F}_\star(\mathrm{M}_n)$}\label{NOFG.sec.5}

A triangulation of $\mathrm{M}_n$ is \emph{simplicial} when it does not contain any loop arc or any two arcs with the same pair of vertices. In other words, the triangles, arcs, and vertices of a simplicial triangulation form a simplicial complex. It has been shown by Edelman and Reiner~\cite{EdelmanReiner1997} that the subgraph $\mathcal{F}_\star(\mathrm{M}_n)$ of $\mathcal{F}(\mathrm{M}_n)$ induced by the simplicial triangulations is connected. In this section, we bound the diameter of $\mathcal{F}_\star(\mathrm{M}_n)$. As mentioned in~\cite{EdelmanReiner1997}, $\mathcal{F}_\star(\mathrm{M}_n)$ is empty when $n$ is at most $4$ and $\mathrm{M}_5$ has a unique simplicial triangulation shown in Figure~\ref{NOFG.sec.5.fig.1} in the cross-cap representation of the M{\"o}bius strip.
\begin{figure}[b]
\begin{centering}
\includegraphics[scale=1]{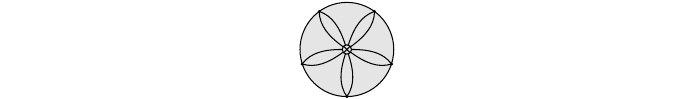}
\caption{The unique simplicial triangulation of $\mathrm{M}_5$.}\label{NOFG.sec.5.fig.1}
\end{centering}
\end{figure}

Our lower bound on the diameter of $\mathcal{F}_\star(\mathrm{M}_n)$ will be a consequence of Theorem~\ref{NOFG.sec.3.thm.1}. In order to  relate the diameters of $\mathcal{F}_\star(\mathrm{M}_n)$ and $\mathcal{F}(\mathrm{M}_n)$, we will show that every triangulation of $\mathrm{M}_n$ is close in $\mathcal{F}_\star(\mathrm{M}_n)$ to a simplicial triangulation. Assume that $n$ is at least $5$ and consider a triangulation $T$ of $\mathrm{M}_n$. Observe first that, when the interior arcs of $T$ are not all non-separating, there is always a flip that replaces a a separating arc by a non-separating arc. In order to see that, let us cut away all the triangles of $T$ that are not incident to a non-separating arc. Since $T$ contains at least one non-separating arc, this results in a triangulation of $\mathrm{M}_{n-k}$ where $k$ is the number of separating arcs in $T$. At least one of the boundary arcs of $\mathrm{M}_{n-k}$ is a separating interior arc of $T$. Flipping that arc introduces a non-separating arc in $T$. This proves the following.

\begin{prop}\label{NOFG.sec.5.prop.1}
When $n$ is at least $5$, any triangulation of $\mathrm{M}_n$ is at most four flips away from a triangulation with at least five non-separating arcs
\end{prop}

Now recall that if a triangulation of $\mathrm{M}_n$ contains a separating loop, then its only non-separating arc is another loop as shown on the left of Figure~\ref{NOFG.sec.1.fig.1}. Therefore, a triangulation $T$ of $\mathrm{M}_n$ that contains at least two non-separating arcs has at most one loop and that loop is one of the non-separating arcs of $T$. If $T$ contains a non-separating loop twice incident to a point $u$ and at least two other non-separating arcs, then the two triangles of $T$ incident to the loop must as shown in Figure~\ref{NOFG.sec.5.fig.2} where $v$ and $w$ are different vertices. 
\begin{figure}[b]
\begin{centering}
\includegraphics[scale=1]{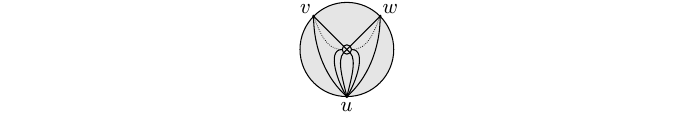}
\caption{The two triangles incident to a non-separating loop in a triangulation of $\mathrm{M}_n$ and the arc introduced when flipping that loop represented as a dotted line.}\label{NOFG.sec.5.fig.2}
\end{centering}
\end{figure}
In particular, flipping the loop introduces the non-separating arc shown as a dotted line in the figure. Since that arc is not a loop, this proves the following.

\begin{prop}\label{NOFG.sec.5.prop.2}
If a triangulation $T$ of $\mathrm{M}_n$ contains at least three non-separating arcs one of whose is a loop, then $T$ is a single flip away from a loopless triangulation with the same number of non-separating arcs.
\end{prop}

Provided that $n$ is at least $5$, we now show that every triangulation of $\mathrm{M}_n$ is a constant number of flips away from a simplicial triangulation.

\begin{lem}\label{NOFG.sec.5.lem.1}
If $n$ is at least $5$, then any triangulation of $\mathrm{M}_n$ is at most seven flips away from a simplicial triangulation of $\mathrm{M}_n$.
\end{lem}
\begin{proof}
Assume that $n$ is at least $5$ and consider a triangulation $T$ of $\mathrm{M}_n$. According to Propositions~\ref{NOFG.sec.5.prop.1} and~\ref{NOFG.sec.5.prop.2}, $T$ is at most $4$ flips away in $\mathcal{F}(\mathrm{M}_n)$ from a triangulation $T'$ with no loop and at least five non-separating arcs. If $T'$ is not simplicial, it contains at least one pair of arcs $\alpha$ and $\beta$ with the same two distinct vertices $a$ and $b$. Observe that $\alpha$ or $\beta$ must be non-separating. Indeed up to isotopy, there are precisely two different separating or boundary arcs with vertices $a$ and $b$ in $\mathrm{M}_n$ and together, they bound a subsurface $\mathrm{M}_2$ within $\mathcal{F}(\mathrm{M}_n)$. However, all the triangulations of $\mathrm{M}_2$ contain a loop arc whereas $T'$ does not and therefore $\alpha$ and $\beta$ cannot both be either a separating or a boundary arc. As there is only one non-separating arc between $a$ and $b$ up to isotopy, $\alpha$ is a separating or a boundary arc while $\beta$ is a non-separating arc.
\begin{figure}[b]
\begin{centering}
\includegraphics[scale=1]{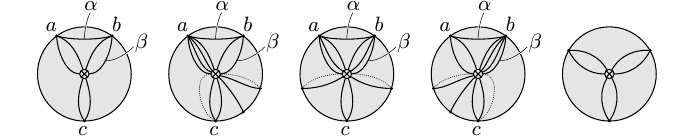}
\caption{The triangulation $T'$ (left), a flip in this triangulation shown as a function of the number of non-separating arcs incident to $a$ and $b$ (center), and a triangulation with three pairs of arcs with the same two vertices (right).}\label{NOFG.sec.5.fig.3}
\end{centering}
\end{figure}

The triangulation $T'$ is then as shown on the left of Figure~\ref{NOFG.sec.5.fig.3}. Note that, in this figure, the point $c$ is different from $a$ and $b$ because otherwise $T'$ would contain a loop. Further note that all the non-separating arcs of $T'$ must be incident to $a$ or $b$. As a consequence, flipping $\beta$ in $T'$ introduces a non-separating arc that is not a loop and such that no other arc of $T'$ has the same pair of vertices as sketched at the center of Figure~\ref{NOFG.sec.5.fig.3} depending on whether $a$, $b$, or neither of these two points are incident to exactly two non-separating arcs in $T'$. Finally, observe that a triangulation of $\mathrm{M}_n$ contains at most three different pairs of arcs with the same two vertices and that it contains exactly three such pairs when the triangles incident to the non-separating arcs in these pairs form the triangulation of $\mathrm{M}_3$ shown on the right of Figure~\ref{NOFG.sec.5.fig.3}. Such a triangulation has just three non-separating arcs. As $T'$ contains at least five non-separating arcs, at most two pairs of arcs in $T'$ can have the same two vertices. Therefore, $T'$ is at most two flips away from a simplicial triangulation.
\end{proof}

We can now lower bound the diameter of $\mathcal{F}_\star(\mathrm{M}_n)$ using Lemma~\ref{NOFG.sec.5.lem.1} and the lower bound from Theorem~\ref{NOFG.sec.3.thm.1} on the diameter of $\mathcal{F}(\mathrm{M}_n)$.

\begin{thm}\label{NOFG.sec.5.thm.1}
For all $n$ at least $5$,
$$
\mathrm{diam}(\mathcal{F}_\star(\mathrm{M}_n))\geq\biggl\lfloor\frac{5}{2}n\biggr\rfloor-16\mbox{.}
$$
\end{thm}
\begin{proof}
By Lemma~\ref{NOFG.sec.5.lem.1}, there exists a path of length at most
$$
\mathrm{diam}(\mathcal{F}_\star(\mathrm{M}_n))+14
$$
in $\mathcal{F}(\mathrm{M}_n)$ between any two triangulations of $\mathrm{M}_n$. As a consequence,
$$
\mathrm{diam}(\mathcal{F}(\mathrm{M}_n))\leq\mathrm{diam}(\mathcal{F}_\star(\mathrm{M}_n))+14
$$
and the result follows from Theorem~\ref{NOFG.sec.3.thm.1}.
\end{proof}

\begin{rem}\label{NOFG.sec.5.rem.1}
Even though Theorem~\ref{NOFG.sec.5.thm.1} is proven using Theorem~\ref{NOFG.sec.3.thm.1}, the boundary arc contractions that we relied on to establish the latter theorem cannot be used directly to lower bound the diameter of $\mathcal{F}_\star(\mathrm{M}_n)$ because they can transform a simplicial triangulation of $\mathrm{M}_n$ into a non-simplicial one.
\end{rem}

Let us turn to upper bounding the diameter of $\mathcal{F}_\star(\mathrm{M}_n)$. In order to do that let us consider a simplicial triangulation $T$ of $\mathrm{M}_n$ and introduce some notions and terminology. Consider a non-separating arc $\alpha$ in $T$ with vertices $a$ and $b$. Cutting $T$ along $\alpha$ results in a triangulation $U$ of a disk $\Delta_{n+2}$ with $n+2$ marked points and two copies of $\alpha$ in the boundary. Since $T$ is simplicial, $\alpha$ is not a loop and therefore, $\Delta_{n+2}$ has two copies of $a$ and two copies of $b$ in its boundary. Moreover, each copy of $a$ or $b$ is consecutive in that boundary to a single copy of the other point via a copy of $\alpha$ as sketched in Figure~\ref{NOFG.sec.5.fig.4}.
\begin{figure}[b]
\begin{centering}
\includegraphics[scale=1]{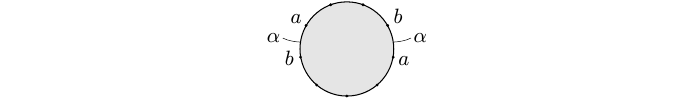}
\caption{The disk $\Delta_{n+2}$.}\label{NOFG.sec.5.fig.4}
\end{centering}
\end{figure}

Now consider a vertex $u$ of $U$ that is neither a copy of $a$ or a copy of $b$. The $k$ triangles of $U$ incident to $u$ collectively form a triangulation of the disk $\Delta_{k+2}$ with $k+2$ marked point in its boundary. This triangulation of $\Delta_{k+2}$ will be called the \emph{star} of $u$ in $U$ and denoted by $\mathrm{star}_U(u)$. Note that all the interior arcs of $\mathrm{star}_U(u)$ are incident to $u$. Again, since $T$ is simplicial, at most one copy of $a$ and one copy of $b$ are vertices of $\mathrm{star}_U(u)$. For the same reason, a boundary arc of $\mathrm{star}_U(u)$ is between a copy of $a$ and a copy of $b$ only when $u$ is a vertex of one of the two triangles of $U$ incident to a copy of $\alpha$. Further observe that a boundary arc of $\mathrm{star}_U(u)$ is between $u$ and a copy of $a$ or $b$ precisely when $u$ is consecutive to that copy in the boundary of $U$. As at most four vertices of $U$ are consecutive to a copy of $a$ or $b$ in the boundary of $U$ and exactly two triangles of $U$ are incident to a copy of $\alpha$, we get the following.

\begin{prop}\label{NOFG.sec.5.prop.3}
At least $n-8$ of the vertices $u$ of $U$ different from both $a$ and $b$ are such that no boundary arc of $\mathrm{star}_U(u)$ is between a copy of $a$ or a copy of $b$ or between $u$ and a copy of either $a$ or $b$.
\end{prop}

We upper bound the diameter of $\mathcal{F}_\star(\mathrm{M}_n)$ using Proposition~\ref{NOFG.sec.5.prop.3}.

\begin{thm}\label{NOFG.sec.5.thm.2}
There exists a constant $K$ such that for every $n$ at least $5$,
$$
\mathrm{diam}(\mathcal{F}_\star(\mathrm{M}_n))\leq4n+K\mbox{.}
$$
\end{thm}
\begin{proof}
Let us pick
$$
K=\max\bigl\{\mathrm{diam}(\mathcal{F}_\star(\mathrm{M}_n))-4n:5\leq{n}\leq76\bigr\}\mbox{.}
$$ 

By this choice of $K$, the desired statement is immediate when $n$ is at most $76$. Let us prove it by induction when $n$ is at least $77$. Consider two triangulations $T^-$ and $T^+$ of $\mathrm{M}_n$. Consider a non-separating arc $\alpha^-$ of $T^-$ with (necessarily distinct) vertices $a^-$ and $b^-$. Further consider the triangulation $U^-$ obtained by cutting $T^-$ along $\alpha^-$. Denote by $\mathcal{S}^-$ the set of the points $u$ in $\mathcal{P}\mathord{\setminus}\{a^-,b^-\}$ such that no boundary arc of $\mathrm{star}_{U^-}(u)$ is between a copy of $a^-$ or a copy of $b^-$ or between $u$ and a copy of either $a^-$ or $b^-$. Likewise, consider a non-separating arc $\alpha^+$ of $T^+$. Denote by $a^-$ and $a^+$ the vertices of that arc and consider the triangulation $U^+$ and the subset $\mathcal{S}^+$ of $\mathcal{P}\mathord{\setminus}\{a^+,b^+\}$ obtained from $U^+$ as $\mathcal{S}^-$ is obtained from $U^-$. According to Proposition~\ref{NOFG.sec.5.prop.3},
\begin{equation}\label{NOFG.sec.5.thm.2.eq.1}
\bigl|\mathcal{S}^-\cap\mathcal{S}^+\bigr|\geq{n-16}.
\end{equation}

As $U^-$ and $U^+$ have $n-1$ interior arcs, the number of incidences between a vertex and an interior arc is $2n-2$ in each of these triangulations. As a consequence, there must be the combined number of incidences between a point in $\mathcal{S}^-\cap\mathcal{S}^+$ and an interior arc of either $U^-$ or $U^+$ is at most
$$
\frac{4n-4}{\bigl|\mathcal{S}^-\cap\mathcal{S}^+\bigr|}\mbox{.}
$$

Hence, it follows from (\ref{NOFG.sec.5.thm.2.eq.1}) that the combined number of arcs of $U^-$ and $U^+$ incident to a point from $\mathcal{S}^-\cap\mathcal{S}^+$ is on average at most
$$
\frac{4n-4}{n-16}
$$
which is less than $5$ because $n$ is at least $77$. As a consequence, $\mathcal{S}^-\cap\mathcal{S}^+$ must contain a point $u$ such that the combined number of interior arcs of $U^-$ and $U^+$ incident to $u$ is at most $4$. We will now flip all the interior arcs of $\mathrm{star}_{U^-}(u)$ within $T^-$ and all the interior arcs of $\mathrm{star}_{U^+}(u)$ within $T^+$ in order to make $u$ an ear of the resulting two triangulations of $\mathrm{M}_n$.

If $\mathrm{star}_{U^-}(u)$ admits a copy of $a$ or a copy of $b$ as vertices, no boundary arc of $\mathrm{star}_{U^-}(u)$ is between these copies or between $u$ and one of these copies because $u$ belongs to $\mathcal{S}^-$. In particular, the vertices of $\mathrm{star}_{U^-}(u)$ that are copies of $a$ or $b$ must be incident to a (single) interior arc of $\mathrm{star}_{U^-}(u)$. Moreover, flipping these possible two interior arcs of $\mathrm{star}_{U^-}(u)$ first (and in any order) in $T^-$ does not introduce an arc incident to $u$, $a$, or $b$. After these flips, further flipping the remaining interior arcs of $\mathrm{star}_{U^-}(u)$ in any order will result in a triangulation $V^-$ with $u$ as an ear and none of these flips introduces an arc incident to $a$ or $b$. As the only possible multiple arcs or loops that can be introduced by flipping interior arcs of $U^-$ within $T^-$ must be incident to $a$ or $b$, this sequence of flips takes place within $\mathcal{F}_\star(\mathrm{M}_n)$ and in particular, $V^-$ is simplicial.

Similarly, we can go within $\mathcal{F}_\star(\mathrm{M}_n)$ from $T^+$ to a triangulation $V^+$ by flipping each of the interior arcs of $\mathrm{star}_{U^+}(u)$ starting with the arcs incident to $a$ or $b$ if any. Since $V^-$ and $V^+$ are simplicial triangulations that both admit $u$ as an ear, their distance in $\mathcal{F}_\star(\mathrm{M}_n)$ is at most the diameter of $\mathcal{F}_\star(\mathrm{M}_{n-1})$. However, the diameter of $\mathcal{F}_\star(\mathrm{M}_{n-1})$ can be bounded by induction as
$$
\mathrm{diam}(\mathcal{F}_\star(\mathrm{M}_{n-1}))\leq4n-4+K\mbox{.}
$$

As the combined number of interior arcs of $T^-$ and $T^+$ incident to $u$ is at most $4$ and as these interior arcs are precisely the ones that have been flipped to transform $T^-$ and $T^+$ into $V^-$ and $V^+$, this completes the proof.
\end{proof}

Theorem~\ref{NOFG.sec.1.thm.3} is an immediate consequence of Theorems~\ref{NOFG.sec.5.thm.1} and~\ref{NOFG.sec.5.thm.2}.

\noindent{\bf Acknowledgement.} This work was partially funded by the MathSTIC research consortium (CNRS FR3734) from the université Paris 13.

\bibliography{NonOrientableFlipGraphs}
\bibliographystyle{ijmart}

\end{document}